\documentclass[final,leqno,hidelinks]{siamltex1213}
\usepackage{graphicx,amsmath,amsfonts,latexsym,amssymb,hyperref,comment,subeqnarray,appendix}
\usepackage{hyperref,mathrsfs}
\usepackage[mathscr]{euscript}
\usepackage{colortbl}

\newcommand{\cE}{{\mathcal{E}}}
\newcommand{\cT}{{\mathcal{T}}}
\newcommand{\cN}{{\mathcal{N}}}

\newcommand{\tK}{{\tilde {K}}}
\newcommand{\tW}{{\tilde {W}}}
\newcommand{\tH}{{\tilde {H}}}
\newcommand{\bfn}{{\bf n}}
\newcommand{\commentout}[1]{{}}
\newcommand{\abs}[1]{\left|#1\right|}
\newcommand{\norm}[1]{\left\|#1\right\|}
\newcommand{\Frac}[2]{\frac{\textstyle #1}{\textstyle #2}}
\newcommand{\oderiv}[2]{\Frac{d \textstyle #1}{d \textstyle #2 }}
\newcommand{\oderivm}[3]{\Frac{d^{#3}\textstyle #1 }{d \textstyle #2^{#3} } }
\newcommand{\pderiv}[2]{\Frac{\partial \textstyle #1}{\partial \textstyle #2 }}

\newtheorem{remark}{Remark}[section]

\title{Partially Penalized Immersed Finite Element Methods
For Elliptic Interface Problems
\thanks{This work is partially supported by NSF grant DMS-1016313, GRF grant of Hong Kong (Project No. PolyU 501012)} }
\date{}

\author{
Tao Lin\thanks{Department of Mathematics, Virginia Tech, Blacksburg, VA 24061, ({\tt tlin@math.vt.edu}).}
\and
Yanping Lin\thanks{Department of Applied Mathematics, Hong Kong
polytechnic University, Hung Hom, Hong Kong, and Department of
Mathematical and Statistics Science, University of Alberta, Edmonton
AB, T6G 2G1, Canada, ({\tt  yanlin@ualberta.ca}).}
\and
Xu Zhang\thanks{Department of Mathematics,
Virginia Tech, Blacksburg, VA 24061, ({\tt xuz@vt.edu}).}
}
\begin{document}
\maketitle

\begin{abstract}
This article presents new immersed finite element (IFE) methods for solving the popular second order elliptic interface problems on structured Cartesian meshes even if the involved interfaces have nontrivial geometries. These IFE methods contain extra stabilization terms introduced only at interface edges for penalizing the discontinuity in IFE functions. With the enhanced stability due to the added penalty, not only these IFE methods can be proven to have the optimal convergence rate
in an energy norm provided that the exact solution has sufficient regularity, but also numerical results indicate that their convergence rates in both the $H^1$-norm and the $L^2$-norm do not deteriorate when the mesh becomes finer which is a shortcoming of the classic IFE methods in some situations. Trace inequalities are established for both linear and bilinear IFE functions that are not only critical for the error analysis of these new IFE methods, but also are of a great potential to be useful in error analysis for other related IFE methods.
\end{abstract}



\section{Introduction}

Without loss of generality, we consider a domain $\Omega$ that is a union of rectangular domains in $\mathbb{R}^2$, and we assume that $\Omega$ is formed by two different materials separated by a curve $\Gamma$. In particular, this means $\Gamma$ separates $\Omega$ into two sub-domains $\Omega^-$ and $\Omega^+$ such that $\overline{\Omega}= \overline{\Omega^- \cup \Omega^+ \cup \Gamma}$. Consequently, the diffusion coefficient $\beta$ on $\Omega$ is assumed to be a piecewise constant function:
\begin{eqnarray*}
\beta(x,y) = \left\{ \begin{array} {ll}
\beta^-, ~~(x,y) \in \Omega^-, \\
\beta^+, ~~(x,y) \in \Omega^+,
\end{array} \right.
\end{eqnarray*}
such that $\min\{\beta^-, \beta^+\} > 0$.

The main purpose of this article is to present a group of partially penalized immersed finite element (IFE) methods using Cartesian meshes to solve popular elliptic interface problems appearing in many applications in the following form:
\begin{eqnarray}
- \nabla \cdot \big( \beta \nabla u(x,y) \big) &=& f(x,y),~~(x,y) \in \Omega^- \cup \Omega^+,
\label{bvp_pde} \\
u(x,y)&=&0,~~(x,y) \in \partial \Omega, \label{eq:bvp_bc}
\end{eqnarray}
together with the jump conditions on the interface $\Gamma \subset \Omega$:
\begin{eqnarray}
\left[u \right]|_\Gamma &=& 0, \label{eq:bvp_int_1} \\
\left[\beta \pderiv{u}{n}\right]|_{\Gamma} &=& 0. \label{eq:bvp_int_2}
\end{eqnarray}
The homogeneous boundary condition \eqref{eq:bvp_bc} is discussed here for simplicity's sake, the method and related analysis can be readily extended to interface problems with a non-homogeneous boundary condition.

A large number of numerical methods based on Cartesian meshes have been introduced for elliptic interface problems. Since Peskin's pioneering work of the immersed boundary method \cite{CPeskin_Blood_Flow}, a variety of methods have been developed in finite difference formulation, such as the immersed interface method \cite{RLeveque_ZLLi_IIM}, the matched interface and boundary method \cite{YZhou_SZhao_MFeig_GWei_MIB_Elliptic}, and the ghost fluid method \cite{Osher_Ghost_Fluid_1999}. We refer to the book \cite{ZLi_KIto_IIM} for an overview of different numerical methods in finite difference framework.

In finite element formulation, certain types of modifications need to be executed for elements around the interface. One way is to modify the weak formulation of finite element equations near the interface. We refer to some representative methods such as the penalty finite element method
\cite{Babuska_Elliptic_Discontinuous,
JBarrett_CElliott_Fitted_Unfitted},
the unfitted finite element method
\cite{Hansbo_Hansbo_FEM_Elliptic},
the discontinuous Galerkin formulation methods
\cite{PBastin_CEngwer_Unfitted_DG, GGuyomarch_CLee_KJeon_DG_Elliptic}.
An alternative approach is to modify the approximating functions around the interface, for instance, the general finite element method \cite{Babuska_Banerjee_Osborn_Survey_GFEM, IBabuska_CCaloz_JOsborn_FEM_Elliptic_Rough_Cofficients}, the multi-scale finite element method
\cite{CChu_IGraham_THou_Multiscale_FE_Elliptic_Interface, Efendiev_Hou_Multiscale_FEM}, the extended finite element method \cite{Dolbow_Moes_Belytschko_XFEM}, the partition of unity method  \cite{Babuska_Melenk_PU_FEM, Babuska_Zhang_Partition_of_Unity}, to name just a few.

Immersed finite element (IFE) methods are a particular class of finite element (FE) methods belonging to the second approach mentioned above, and they can solve interface problems with meshes independent of the interface
\cite{SAdjerid_TLin_1D_IDG,
SAdjerid_TLin_1D_IFE,
XHe_Thesis_Bilinear_IFE,
XHe_TLin_YLin_Bilinear_Approximation,
XHe_TLin_YLin_XZhang_Moving_CNIFE,
RKafafy_TLin_YLin_JWang_3D_IFE_Electric,
ZLi_IIM_FE,
ZLi_TLin_XWu_Linear_IFE,
TLin_YLin_WSun_ZWang_IFE_4TH,
TLin_XZhang_Elasticity,
SSauter_RWarnke_Composite_FE,
SVallaghe_TPapadopoulo_TriLinear_IFE}. If desired, an IFE method can use a Cartesian mesh to solve a boundary value problem (BVP) whose coefficient is discontinuous
across a curve $\Gamma$ with a non-trivial geometry. The basic idea of an IFE method is to employ standard FE functions in non-interface elements not intersecting with the interface $\Gamma$, but on each interface element, it uses IFE functions constructed with piecewise polynomials based on the natural partition of this element formed by the interface and the jump conditions required by the interface problem. The IFE functions are macro-elements \cite{DBraess_Finite_Element,Clough_Tocher_FE}, and each IFE function partially solves the related interface problem because it satisfies the interface jump conditions in a certain sense. Also, the IFE space on an interface element is consistent with
the corresponding FE space based on the same polynomial space in the sense that the IFE space becomes the FE space if the discontinuity in the coefficient
$\beta$ disappears in that element, see \cite{XHe_Thesis_Bilinear_IFE, XHe_TLin_YLin_Bilinear_Approximation} for more details.

IFE methods have been developed for solving interface problems involving several important types of partial differential equations, such as the second order elliptic equation
\cite{SAdjerid_TLin_1D_IDG,
SAdjerid_TLin_1D_IFE,
YGong_BLi_ZLi_Nonhomo_IFE,
XHe_Thesis_Bilinear_IFE,
XHe_TLin_YLin_Bilinear_Approximation,
RKafafy_TLin_YLin_JWang_3D_IFE_Electric,
Kwak_Wee_Chang_Broken_P1_IFE,
ZLi_IIM_FE,
ZLi_TLin_XWu_Linear_IFE,
TLin_YLin_RRogers_MRyan_Rectangle,
SSauter_RWarnke_Composite_FE,
SVallaghe_TPapadopoulo_TriLinear_IFE,
Wu_Li_Lai_Adaptive_IFE,
XZhang_PHDThesis}, the bi-harmonic and beam equations \cite{TLin_YLin_WSun_ZWang_IFE_4TH}, the planar elasticity system
\cite{YGong_ZLLi_Elas_IFE,
ZLLi_XZYang_IFE_Elasticity,
TLin_DSheen_XZhang_RQ1_IFE_Elasiticity,
TLin_XZhang_Elasticity},
the parabolic equation with fixed interfaces
\cite{Attanayake_Senaratne_Convergence_IFE_Parabolic,
TLin_DSheen_IFE_Laplace,
Wang_Wang_Yu_Immersed_EL_Interfaces}, and the parabolic equation with a moving interface
\cite{XHe_TLin_YLin_XZhang_Moving_CNIFE,
TLin_YLin_XZhang_MoL_Nonhomo,
TLin_YLin_XZhang_IFE_MoL}.
When jump conditions are suitably employed in the construction of IFE functions for an interface problem, the resulting IFE space usually has the optimal approximation capability from the point view of polynomials used in this IFE space
\cite{SAdjerid_TLin_1D_IDG,
BCamp_TLin_YLin_WSun_Quadratic_IFE,
BCamp_Thesis,
XHe_TLin_YLin_Bilinear_Approximation,
ZLi_TLin_YLin_RRogers_linear_IFE,
MBenRomdhane_Thesis_Quadratic_IFE,
XZhang_PHDThesis}. Numerical examples
\cite{SAdjerid_TLin_1D_IFE,
ZLi_IIM_FE,
ZLi_TLin_YLin_RRogers_linear_IFE,
ZLi_TLin_XWu_Linear_IFE} demonstrate that methods based on IFE spaces can converge optimally for second order elliptic interface problems. However, the proof for their optimal error bounds is {\em still elusive} except for the one dimensional case
\cite{SAdjerid_TLin_1D_IFE}, even though there have been a few attempts \cite{Chou_Kwak_Wee_IFE_Triangle_Analysis,XHe_TLin_YLin_Convergence_IFE,Kwak_Wee_Chang_Broken_P1_IFE,Wang_Chen_IFE_Analysis}. For two dimensional elliptic interface problems, only a {\em suboptimal} convergence in the $H^1$-norm has been rigorously proven \cite{XHe_TLin_YLin_Convergence_IFE}.

One of the major obstacles is the error estimation on edges between two interface elements where IFE functions have discontinuity.
Certain types of trace inequalities are needed and can be established, but it is not clear whether the generic constant factor in these inequalities is actually independent of the interface location. The scaling argument in the standard finite element error estimation is not applicable here because the local IFE spaces on two different interface elements are not affine equivalent in general. Besides, numerical experiments have demonstrated that the classic IFE methods in the literature
often have a much lager point-wise error over interface elements which, we believe, is caused by the inter-element discontinuity of IFE functions. In some cases, the convergence rates can even deteriorate when the mesh becomes finer. These observations motivate us to apply a certain penalty over interface edges for controlling  negative impacts from this discontinuity. Natural candidates are those well known penalty strategies for handling inter-element discontinuity in interior penalty Galerkin methods and discontinuous Galerkin methods
\cite{Babuska_penalty,
Babuska_Zlamal_Nonconform_FEM_Penalty,
Brezzi_Cockburn_Marini_Suli_DG,
Douglas_Dupont_Penalty_Elliptic_Parabolic,
OdenBabuskaBaumann_DG_hp,
RiviereWheelerGiraut_DG,
RustenVassilevskiWinther_interior_penalty,
M.F.Wheeler_colloc_interior_penalty}. These considerations lead to the partially penalized IFE methods in this article. Theoretically, thanks to the enhanced stability by the penalty terms, we are able to prove that these new IFE methods do converge optimally in an energy norm. In addition, we have observed through abundant numerical experiments that these partially penalized IFE methods maintain their expected convergence rate in both $H^1$-norm and $L^2$-norm when their mesh becomes finer and finer while the classic IFE methods cannot maintain in some situations.

The partial penalty idea has also been used in the unfitted finite element method \cite{Hansbo_Hansbo_FEM_Elliptic}.
In this method, penalty terms are introduced on interface instead of interface edges because approximating functions are allowed to be discontinuous inside interface elements but they are continuous on element boundaries within each subdomain. IFE methods reverse this idea by imposing continuity of approximating functions inside each element but allowing discontinuity possibly only across interface edges. In addition, on the same mesh, the unfitted finite element method has a slightly larger number of degrees of freedom than IFE methods. On the other hand, the unfitted finite element method has been proven to have the optimal convergence rate under the usual piecewise $H^2$ regularity
\cite{Hansbo_Hansbo_FEM_Elliptic} while the analysis in the represent article needs to assume a piecewise $H^3$ or $W^{2,\infty}$ regular
in order to establish the optimal convergence for the partially penalized IFE methods.

Also, we note that these partially penalized IFE methods and their related error analysis can be readily modified to obtain IFE methods based on the discontinuous Galerkin formulation with advantages such as adaptivity even with Cartesian meshes. However, on the same mesh, the DG IFE methods generally have far more global degrees of freedom. For instance, on a Cartesian triangular mesh, a DG IFE method has about 6 times more unknowns than the classic IFE method. The partially penalized IFE methods presented here have the same global degrees of freedom as their classic counterparts; hence they can be more competitive in applications where advantages of DG IFE methods are not needed.

%

The rest of this article is organized as follows. In Section 2, we derive partially penalized IFE methods based on either linear or bilinear IFE functions for the interface problem. In Section 3, we show that the well-known trace inequalities on an element are also valid for linear and bilinear IFE functions even though they are not $H^2$ functions locally in an interface element. In Section 4, we show that these IFE schemes do have the optimal convergence rate in an energy norm.
In Section 5, we will present numerical examples to demonstrate features of these IFE methods.

\section{Partially penalized IFE methods}

Let $\cT_h$, $0 < h < 1$, be a family of Cartesian triangular or rectangular meshes on $\Omega$. For each mesh $\cT_h$, we let $\cN_h$ be the set of vertices of its elements, and let $\cE_h$ be the set of its edges and ${\mathring \cE}_h$ be the set of interior edges.
In addition, we let $\cT_h^i$ be the set of interface elements of $\cT_h$ and $\cT_h^n$ be the set of non-interface elements. Similarly, we let ${\mathring \cE}_h^i$ be the set of interior interface edges and let ${\mathring \cE}_h^n$ be the set of interior non-interface edges. For every interior edge $B \in {\mathring \cE}_h$, we denote
two elements that share the common edge $B$ by $T_{B,1}$ and $T_{B,2}$. For a function $u$ defined on $T_{B,1}\cup T_{B,2}$, we denote its average and
jump on $B$ by
\begin{align*}
\{u\}_B = \frac{1}{2}\big((u|_{T_{B,1}})|_{B} +  (u|_{T_{B,2}})|_{B}\big),~~[u]_B = (u|_{T_{B,1}})|_{B} -  (u|_{T_{B,2}})|_{B}.
\end{align*}
For simplicity's sake, we will often drop the subscript $B$ from these notations if there is no danger to cause any confusions. We will also use the following function spaces:
\begin{equation*}
  \tW^{r, p}(\Omega) = \{v \in W^{1,p}(\Omega)~|~u|_{\Omega^s} \in W^{r,p}(\Omega^s),~s = + \text{~or~} -\}
\text{~~for $r \geq 1$ and $1 \leq p \leq \infty$},
\end{equation*}
equipped the norm
\begin{equation*}
  \norm{v}_{\tW^{r,p}(\Omega)}^p = \norm{v}_{W^{r,p}(\Omega^-)}^p + \norm{v}_{W^{r,p}(\Omega^+)}^p,~~\forall v \in \tW^{r, p}(\Omega).
\end{equation*}
As usual, for $p=2$, we use $\tH^r(\Omega)= \tW^{r,2}(\Omega)$ and denote its corresponding norm by
\begin{equation*}
  \norm{v}_{r}^2 = \norm{v}_{\tH^r(\Omega)}^2 = \norm{v}_{H^r(\Omega^-)}^2 + \norm{v}_{H^r(\Omega^+)}^2,~~\forall v \in \tH^r(\Omega).
\end{equation*}

With a suitable assumption about the regularity of $\Gamma$ and $f$ (e.g. \cite{Babuska_Elliptic_Discontinuous}), we can assume that the exact solution $u$ to the interface problem is in $\tH^2(\Omega)$.
To derive a weak form of interface problem described by \eqref{bvp_pde}-\eqref{eq:bvp_int_2} for an IFE method, we will use the following space:
\begin{equation*}
  V_h = \{v ~|~v \text{~satisfies conditions (HV1)-(HV4) described as follows}\}
\end{equation*}
\begin{description}
  \item[(HV1)] $v|_K \in H^1(K),~\forall K \in \cT_h$.
  \item[(HV2)] $v$ is continuous at every $X \in \cN_h$.
  \item[(HV3)] $v$ is continuous across each $B \in {\mathring \cE}_h^n$.
  \item[(HV4)] $v|_{\partial \Omega} = 0$.
\end{description}

We multiply equation \eqref{bvp_pde} by a test function $v \in V_h$, integrate both sides on each element $K \in \cT_h$, and apply Green's formula to have
\begin{equation*}
  \int_K \beta \nabla v \cdot \nabla u dX - \int_{\partial K} \beta \nabla u \cdot \bfn vds = \int_K vf dX.
\end{equation*}
Summarizing over all elements leads to
\begin{equation}\label{eq:weak_1}
  \sum_{K \in \cT_h} \int_K \beta \nabla v \cdot \nabla u dX - \sum_{B \in {\mathring \cE}_h^i} \int_B \left\{\beta \nabla u \cdot \bfn_B\right\}[v] ds = \int_\Omega v f dX.
\end{equation}
Here we have used the fact that
\begin{equation*}
\left\{\beta \nabla u \cdot \bfn_B\right\}_B = (\beta \nabla u \cdot \bfn_B)|_B, ~~\forall B \in {\mathring \cE}_h^n.
\end{equation*}
Because of the regularity of $u$, for arbitrary parameters $\epsilon, \alpha > 0$, and $\sigma_B^0\geq 0$, we have
\begin{equation}\label{eq:weak_2}
\epsilon \sum_{B \in {\mathring \cE}_h^i}\int_B \left\{\beta \nabla v \cdot \bfn_B\right\} [u] ds = 0,
~\sum_{B \in {\mathring \cE}_h^i}\int_B \frac{\sigma_B^0}{\abs{B}^\alpha} [v][u] ds = 0.
\end{equation}
Therefore, adding \eqref{eq:weak_2} to \eqref{eq:weak_1} leads to the following weak form of the interface problem \eqref{bvp_pde}-\eqref{eq:bvp_int_2}:

\begin{eqnarray}
   \sum_{K \in \cT_h} \int_K \beta \nabla v \cdot \nabla u dX - \sum_{B \in {\mathring \cE}_h^i} \int_B \left\{\beta \nabla u \cdot \bfn_B\right\}[v] ds \label{eq:weak_form} &&  \\
   +\epsilon \sum_{B \in {\mathring \cE}_h^i}\int_B \left\{\beta \nabla v \cdot \bfn_B\right\} [u] ds
   + \sum_{B \in {\mathring \cE}_h^i}\int_B \frac{\sigma_B^0}{\abs{B}^\alpha} [v][u] ds &=& \int_\Omega v f dX,~~\forall v \in V_h.\nonumber
    \end{eqnarray}
We now recall the linear and bilinear IFE spaces to be used in our partially penalized IFE methods based on the weak form \eqref{eq:weak_form}. On each element $K \in \cT_h$, we let
\begin{eqnarray*}
S_h(K) = span\{\phi_j(X), 1 \leq j \leq d_K\}, ~~d_K = \begin{cases}
3,&\text{if $K$ is a triangular element}, \\
4,&\text{if $K$ is a rectangular element},
\end{cases}
\end{eqnarray*}
where $\phi_j, 1 \leq j \leq d_K$ are the standard linear or bilinear nodal basis functions for $K \in \cT_h^n$; otherwise, for $K \in \cT_h^i$,
$\phi_j, 1 \leq j \leq d_K$ are the linear or bilinear IFE basis functions discussed in \cite{ZLi_TLin_YLin_RRogers_linear_IFE,ZLi_TLin_XWu_Linear_IFE} and \cite{XHe_TLin_YLin_Bilinear_Approximation,TLin_YLin_RRogers_MRyan_Rectangle}, respectively. Then, we define the IFE space over the whole solution domain $\Omega$ as follows:
\begin{equation*}
S_h(\Omega) = \{v ~|~\text{$v$ satisfies conditions (IFE1) - (IFE3) given below}\}
\end{equation*}
\begin{description}
  \item[(IFE1)] $v|_K \in S_h(K),~\forall K \in \cT_h$.
  \item[(IFE2)] $v$ is continuous at every $X \in \cN_h$.
  \item[(IFE3)] $v|_{\partial \Omega} = 0$.
\end{description}
%
%
%

It is easy to see that $S_h(\Omega) \subset V_h(\Omega)$. Now, we describe the partially penalized IFE methods for the interface problem \eqref{bvp_pde}-\eqref{eq:bvp_int_2}: find $u_h \in S_h(\Omega)$ such that
\begin{equation}\label{eq:IFE_eq}
a_h(v_h, u_h) = (v_h, f),~~\forall v_h \in S_h(\Omega),
\end{equation}
where the bilinear form $a_h(\cdot, \cdot)$ is defined on $S_h(\Omega)$ by
\begin{eqnarray}
 a_h(v_h, w_h) &=& \sum_{K \in \cT_h} \int_K \beta \nabla v_h \cdot \nabla w_h dX - \sum_{B \in {\mathring \cE}_h^i} \int_B \left\{\beta \nabla w_h \cdot \bfn_B\right\}[v_h] ds \nonumber \\
 &+& \epsilon \sum_{B \in {\mathring \cE}_h^i}\int_B \left\{\beta \nabla v_h \cdot \bfn_B\right\} [w_h] ds
+ \sum_{B \in {\mathring \cE}_h^i}\int_B \frac{\sigma_B^0}{\abs{B}^\alpha} [v_h][w_h] ds,~~\forall v_h, w_h \in S_h(\Omega).
\label{eq:IFE_BF}
\end{eqnarray}


\section{Trace inequalities for IFE functions}

Using the standard scaling argument, we can obtain the following well known trace inequalities \cite{Riviere_DG_book}: there exists a constant $C$ such that
\begin{eqnarray}
  \quad\norm{v}_{L^2(B)} &\leq& C \abs{B}^{1/2}\abs{K}^{-1/2}\left(\norm{v}_{L^2(K)} + h \norm{\nabla v}_{L^2(K)}\right),~\forall v \in H^1(K),  \label{eq:trace_inq_1}\\
  \quad\norm{\nabla v}_{L^2(B)} &\leq& C \abs{B}^{1/2}\abs{K}^{-1/2}\left(\norm{\nabla v}_{L^2(K)} + h \norm{\nabla^2 v}_{L^2(K)}\right),~\forall v \in H^2(K). \label{eq:trace_inq_2}
\end{eqnarray}
where $B$ is an edge of $K$.

Our goal in this section is to extend these trace inequalities to IFE functions in $S_h(K)$ for $K \in \cT_h^i$. First, we recall that $S_h(K) \subset C(K)\cap H^1(K)$ for all $K \in \cT_h$ \cite{XHe_TLin_YLin_Bilinear_Approximation, ZLi_TLin_YLin_RRogers_linear_IFE}. This implies that inequality \eqref{eq:trace_inq_1} is also valid for
$v \in S_h(K)$ even if $K \in \cT_h^i$. However, the second trace inequality \eqref{eq:trace_inq_2}
 cannot be applied to $v \in S_h(K)$ with $K \in \cT_h^i$ because $v \not \in H^2(K)$ in general.

\subsection{Trace inequalities for linear IFE functions}

It is relatively easier to prove that the trace inequality for a linear IFE function in a triangular interface element is true because its gradient is a piecewise constant function. Without loss of generality, we consider the following triangular interface element
\begin{equation*}
K = \bigtriangleup A_1A_2A_3, ~~~A_1 = (0,0),~A_2 = (h,0), ~A_3 = (0,h).
\end{equation*}
Assume that the interface $\Gamma$ intersects the edge of $K$ at points $D$ and $E$
and the straight line $\overline{DE}$ separates $K$ into $K^-$ and $K^+$, see the illustration on the left in Fig. \ref{fig:tri_rec_IFE_elements}. Consider a linear IFE function on $K$ in the following form
\begin{equation}
v(x,y) = \begin{cases}
v^-(x,y) = c_1^- + c_2^-x + c_3^-y,& \text{if~} (x, y) \in K^-, \\
v^+(x,y) = c_1^+ + c_2^+x + c_3^+y,& \text{if~} (x, y) \in K^+,
\end{cases} \label{eq:linear_IFE_c1c2c3_format}
\end{equation}
which satisfies the following jump conditions \cite{ZLi_TLin_YLin_RRogers_linear_IFE}:
\begin{align} \label{eq:linear_IFE_jump_cond_coef_one_piece_bnd_by_another}
v^-(D) = v^+(D), ~~v^-(E) = v^+(E), ~~\beta^- \pderiv{v^-}{\bfn_{\overline{DE}}} = \beta^+ \pderiv{v^+}{\bfn_{\overline{DE}}}.
\end{align}

\begin{figure}[hbt]
\centerline{
\hbox{\includegraphics[height=1.5in]{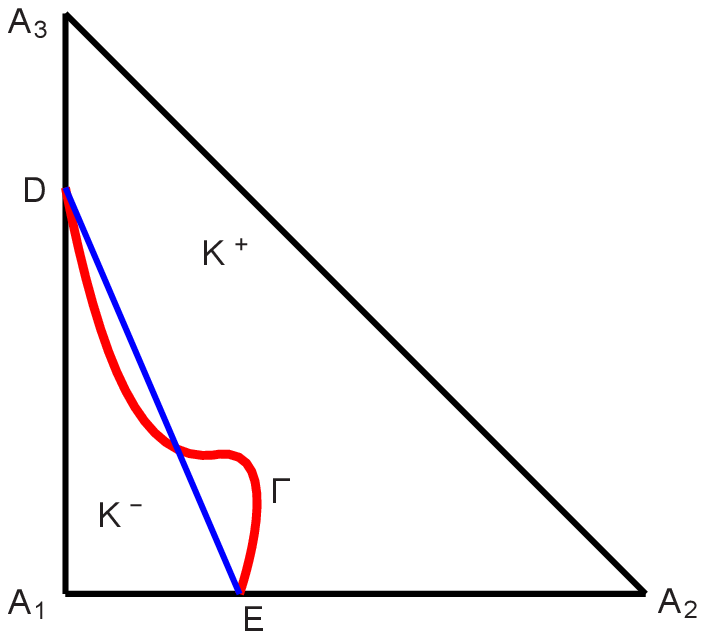}}~~~~
\hbox{\includegraphics[height=1.5in]{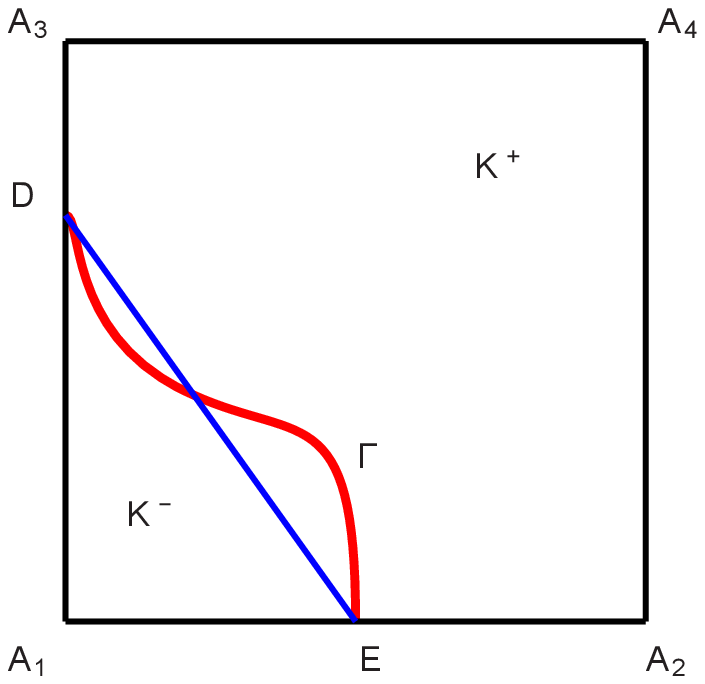}}
}
\caption{a triangular interface element (left) and a rectangular interface element (right)}
\label{fig:tri_rec_IFE_elements}
\centering
\end{figure}

\begin{lemma}\label{lem:linear_IFE_coef_one_piece_bnd_by_another}
There exists a constant $C>1$ independent of the interface location such that for every linear IFE function $v$ on the interface element $K$ defined in \eqref{eq:linear_IFE_c1c2c3_format} the following inequalities hold:
\begin{equation}\label{eq:linear_IFE_coef_one_piece_bnd_by_another}
  \frac{1}{C} \norm{(c_1^+, c_2^+, c_3^+)}\leq \norm{(c_1^-, c_2^-, c_3^-)} \leq C \norm{(c_1^+, c_2^+, c_3^+)}.
\end{equation}
\end{lemma}
\begin{proof}
We prove the second inequality in \eqref{eq:linear_IFE_coef_one_piece_bnd_by_another}, and similar arguments
can be used to show the first one.  Applying the jump conditions
\eqref{eq:linear_IFE_jump_cond_coef_one_piece_bnd_by_another} we can show that coefficients of
$v(x,y)$ must satisfy the following equality:
\begin{align*}
M^- \begin{pmatrix}
c_1^- \\
c_2^- \\
c_3^-
\end{pmatrix} = M^+ \begin{pmatrix}
c_1^+ \\
c_2^+ \\
c_3^+
\end{pmatrix},
\end{align*}
where $M^s, ~s = -, +$ are two matrices.
Without loss of generality, we further assume
\begin{align*}
D = (0,dh), E = (eh, 0), \text{~with~} 0 \leq d \leq 1, 0 \leq e \leq 1.
\end{align*}
Then
\begin{eqnarray*}
M^s = \begin{pmatrix}
1&0&dh \\
1&eh&0 \\
0&-\beta^s dh & -\beta^s eh
\end{pmatrix},~~s = - \text{~or~} +
\end{eqnarray*}
whose determinant is
\begin{equation*}
\det(M^s) = -\beta^s(d^2+e^2)h^2,~~s = - \text{~or~} +
\end{equation*}
which is nonzero because $(d,e) \not = (0,0)$. Hence we can solve for $c_i^+$s in terms of $c_i^-$ to have
\begin{align}
&c_i^+ = f_{i1}c_1^- + f_{i2}c_2^- + f_{i3}c_3^-,~~i = 1, 2, 3,
\label{eq:linear_IFE_coef_one_piece_bnd_by_another_3} \\
&f_{ij} = \Frac{g_{ij}^-\beta^-}{\beta^+(d^2+e^2)} + \frac{g_{ij}^+\beta^+}{\beta^+(d^2+e^2)}, ~~1 \leq i, j \leq 3, \label{eq:linear_IFE_coef_one_piece_bnd_by_another_4}
\end{align}
with
\begin{align*}
&g_{11}^- = 0, g_{11}^+ = d^2+e^2, ~g_{12}^- = -d^2eh, g_{12}^+ = d^2eh, ~g_{13}^- = -de^2h, g_{13}^+ = de^2h, \\
&g_{21}^- = 0, g_{21}^+ = 0, ~g_{22}^- = d^2, g_{22}^+ = e^2, ~g_{23}^- = de, g_{23}^+ = -de, \\
&g_{31}^- = 0, g_{31}^+ = 0, ~g_{32}^- = de, g_{32}^+ = -de, ~g_{33}^- = e^2, g_{33}^+ = d^2.
\end{align*}
Therefore, there exists a constant $C$ that depends on $\beta^-$ and $\beta^+$, but is independent of $d, e$, such that
\begin{align*}
\abs{f_{ij}} \leq C, ~1 \leq i,j \leq 3.
\end{align*}
Then, the second inequality in \eqref{eq:linear_IFE_coef_one_piece_bnd_by_another} follows from
\eqref{eq:linear_IFE_coef_one_piece_bnd_by_another_3} and the above bounds for $\abs{f_{ij}},~1 \leq i,j \leq 3$.
\end{proof}

\begin{remark}\label{rem:linear_IFE_coef_one_piece_bnd_by_another}
In the proof of Lemma \ref{lem:linear_IFE_coef_one_piece_bnd_by_another}, we have shown
$f_{i1} = 0, i = 2, 3$. Consequently, we can show that there exists a constant $C>1$ such that
\begin{equation}\label{eq:linear_IFE_coef_one_piece_bnd_by_another_1_1}
  \frac{1}{C}\norm{(c_2^+, c_3^+)}\leq \norm{(c_2^-, c_3^-)} \leq C\norm{(c_2^+, c_3^+)}.
\end{equation}
\end{remark}
Now, we establish the trace inequality on a triangular interface element $K = \bigtriangleup A_1A_2A_3$.
\begin{lemma}\label{lem:linear_IFE_trace_ineq}
There exists a constant $C$ independent of the interface location such that for every linear IFE function $v$ on $K$ the following
inequalities hold
\begin{align}
&\norm{\beta v_p}_{L^2(B)} \leq C h^{1/2} \abs{K}^{-1/2}\norm{\sqrt{\beta} \nabla v}_{L^2(K)}, ~~p = x,y,  \label{eq:linear_IFE_trace_ineq_1}\\
&\norm{\beta \nabla v \cdot \bfn_B}_{L^2(B)} \leq C h^{1/2} \abs{K}^{-1/2}\norm{\sqrt{\beta} \nabla v}_{L^2(K)}. \label{eq:linear_IFE_trace_ineq_2}
\end{align}
\end{lemma}
\begin{proof}
Without loss of generality, we assume again that the interface $\Gamma$ intersects with the boundary of $K$ at
\begin{align*}
D = (0,dh), E = (eh, 0), \text{~with~} 0 \leq d \leq 1, 0 \leq e \leq 1,
\end{align*}
and the line $\overline{DE}$ separates $K$ into two subelement $K^-$ and $K^+$ with $A_3 \in K^+$ and
$\abs{K^+} \geq \frac{1}{2}\abs{K}$. Furthermore, we assume $B = \overline{A_1A_3}$ is an interface edge with
$B = B^- \cup B^+$. Similar arguments can be applied to establish the trace inequality in other cases.

By direct calculations, we have
\begin{eqnarray*}
  \norm{\beta^+ v_x}_{L^2(B^+)}^2&=& (c_2^+)^2 \abs{B^+}\big(\abs{\beta^+}\big)^2 \leq \big((c_2^+)^2 + (c_3^+)^2\big)\abs{K^+} \frac{\abs{B^+}}{\abs{K^+}}\big(\abs{\beta^+}\big)^2 \nonumber \\
  &=&\beta^+ \frac{\abs{B^+}}{\abs{K^+}} \norm{\sqrt{\beta^+} \nabla v}_{L^2(K^+)}^2 \leq C \frac{\abs{B^+}}{\abs{K^+}} \norm{\sqrt{\beta} \nabla v}_{L^2(K)}^2, \nonumber
\end{eqnarray*}
\emph{i.e.},
\begin{equation}\label{eq:linear_IFE_trace_ineq_3}
  \norm{\beta^+ v_x}_{L^2(B^+)}\leq 2Ch^{1/2}\abs{K}^{-1/2} \norm{\sqrt{\beta} \nabla v}_{L^2(K)}.
\end{equation}
Similarly, we can show that
\begin{equation}\label{eq:linear_IFE_trace_ineq_4}
\norm{\beta^+ v_y}_{L^2(B^+)}
\leq 2Ch^{1/2}\abs{K}^{-1/2} \norm{\sqrt{\beta} \nabla v}_{L^2(K)}.
\end{equation}
On $B^-$, applying the estimates in {Remark} \ref{rem:linear_IFE_coef_one_piece_bnd_by_another}, we have
\begin{align}
\norm{\beta^- v_x}_{L^2(B^-)}^2 &= (c_2^-)^2 \abs{B^-}\big(\abs{\beta^-}\big)^2 \leq C\big((c_2^+)^2 + (c_3^+)^2\big)\abs{K^+} \frac{\abs{B^-}}{\abs{K^+}}\big(\abs{\beta^-}\big)^2 \nonumber \\
&=C\beta^- \frac{\abs{B^-}}{\abs{K^+}} \norm{\sqrt{\beta^+} \nabla v}_{L^2(K^+)}^2 \leq C \frac{\abs{B^-}}{\abs{K^+}} \norm{\sqrt{\beta} \nabla v}_{L^2(K)}^2. \nonumber
\end{align}
Hence,
\begin{equation}\label{eq:linear_IFE_trace_ineq_5}
\norm{\beta^- v_x}_{L^2(B^-)}
\leq 2Ch^{1/2}\abs{K}^{-1/2} \norm{\sqrt{\beta} \nabla v}_{L^2(K)}.
\end{equation}
Similarly,
\begin{equation}
\norm{\beta^- v_y}_{L^2(B^-)}
\leq 2Ch^{1/2}\abs{K}^{-1/2} \norm{\sqrt{\beta} \nabla v}_{L^2(K)}. \label{eq:linear_IFE_trace_ineq_6}
\end{equation}
Then, the combination of \eqref{eq:linear_IFE_trace_ineq_3} and \eqref{eq:linear_IFE_trace_ineq_5}
yields the inequality \eqref{eq:linear_IFE_trace_ineq_1} for $p = x$. Similarly,
the inequality \eqref{eq:linear_IFE_trace_ineq_1} for $p = y$ follows from combining
\eqref{eq:linear_IFE_trace_ineq_4} and \eqref{eq:linear_IFE_trace_ineq_6}. Finally,
\eqref{eq:linear_IFE_trace_ineq_2} follows directly from \eqref{eq:linear_IFE_trace_ineq_1}.
\end{proof}

\subsection{Trace inequalities for bilinear IFE functions}

Without loss of generality, we consider a rectangular interface element $K$ with the following vertices:
\begin{align*}
A_1 = (0,0), ~~A_2 = (h, 0),~~ A_3 = (0,h),~~ A_4 = (h,h).
\end{align*}
Again, assume that the interface $\Gamma$ intersects with $\partial K$ at points
$D$ and $E$ and the linear $\overline{DE}$ separates $K$ into two subelements
$K^-$ and $K^+$, see the illustration on the right in Fig. \ref{fig:tri_rec_IFE_elements}. We assume that
$K$ is one of the two types of rectangular interface elements \cite{XHe_Thesis_Bilinear_IFE, XHe_TLin_YLin_Bilinear_Approximation}:

\noindent
{\bf Type I interface element}: The interface $\Gamma$ intersects $\partial K$ at
\begin{align*}
D = (0, dh), E=(eh, 0), ~~0 \leq d \leq 1, 0 \leq e \leq 1.
\end{align*}
{\bf Type II interface element}: The interface $\Gamma$ intersects $\partial K$ at
\begin{align*}
D = (dh, h), E=(eh, 0), ~~0 \leq d \leq 1, 0 \leq e \leq 1.
\end{align*}
On this interface element $K$, let $v$ be a bilinear IFE function in the following form:
\begin{align}
v(x,y) = \begin{cases}
v^-(x,y) = c_1^- + c_2^-x + c_3^-y + c_4xy, & \text{if~} (x, y) \in K^-, \\
v^+(x,y) = c_1^+ + c_2^+x + c_3^+y + c_4xy, & \text{if~} (x, y) \in K^+,
\end{cases}\label{eq:bilinear_IFE_c1c2c3c4_format}
\end{align}
which satisfies pertinent interface jump conditions \cite{XHe_Thesis_Bilinear_IFE, XHe_TLin_YLin_Convergence_IFE}.
First, using similar arguments, we can show that the coefficients of a bilinear IFE function
$v$ satisfy inequalities similar to those in Lemma \ref{lem:linear_IFE_coef_one_piece_bnd_by_another}.

\begin{lemma}\label{lem:bilinear_IFE_coef_one_piece_bnd_by_another}
There exists a constant $C>1$ independent of the interface location such that for every bilinear IFE function $v$ on the interface element $K$ defined in \eqref{eq:bilinear_IFE_c1c2c3c4_format} the following inequalities hold:
\begin{equation}\label{eq:bilinear_IFE_coef_one_piece_bnd_by_another}
  \frac{1}{C} \norm{(c_1^+, c_2^+, c_3^+, c_4)}\leq\norm{(c_1^-, c_2^-, c_3^-, c_4)} \leq C \norm{(c_1^+, c_2^+, c_3^+, c_4)}.
\end{equation}
\end{lemma}
The proof of trace inequalities for a bilinear IFE function is a little more complicated because its gradient is
not a constant. The following lemma provides an aid.

\begin{lemma}\label{lem:bilinear_IFE_gradient_lower_bnd}
Assume $K$ is an interface element such that
\begin{align*}
\abs{K^s} \geq \frac{1}{2}\abs{K},
\end{align*}
with $s = -$ or $+$. Then there exists a polygon $\tK \subset K^s$ and two positive constants
$C_1$ and $C_2$ independent of the interface location such that
\begin{align}
&\abs{\tK} \geq C_1 \abs{K}, \label{eq:bilinear_IFE_gradient_lower_bnd_1}\\
&\frac{h}{\abs{\tK}} \norm{\sqrt{\beta^s}\nabla v}_{L^2(\tK)}^2 \geq C_2\beta^s
\big(h(c_2^s)^2 + h(c_3^s)^2 + h^3 (c_4)^2\big). \label{eq:bilinear_IFE_gradient_lower_bnd_2}
\end{align}
\end{lemma}
\begin{proof} Let us partition $K$ into $4$ congruent squares $K_i, i = 1, 2, 3, 4$ by the lines connecting the two pairs of opposite mid points of edges of $K$ such that $A_i$ is a vertex of $K_i$. Since $\abs{K^s} \geq \frac{1}{2}\abs{K}$, one of these $4$ small squares must be inside $K^s$. Without loss generality, we assume that $K_4 \subset K^s$. By direct calculations we have
\begin{align*}
&\norm{v_x}_{L^2(K_4)}^2 = \frac{h^2}{48}(12(c_2^s)^2 + 18c_2^sc_4h + 7 c_4^2h^2) \geq \frac{h^2}{48}
\left[\left(12 - \frac{9}{\sigma_1}\right)(c_2^s)^2 + (7-9\sigma_1)c_4^2h^2\right], \\
&\norm{v_y}_{L^2(K_4)}^2 = \frac{h^2}{48}(12(c_3^s)^2 + 18c_3^sc_4h + 7 c_4^2h^2)\geq \frac{h^2}{48}
\left[\left(12 - \frac{9}{\sigma_2}\right)(c_3^s)^2 + (7-9\sigma_2)c_4^2h^2\right],
\end{align*}
where $\sigma_1$ and $\sigma_2$ are arbitrary positive constants. Letting $\sigma_i = \sigma \in (9/12, 7/9), i = 1, 2$ in the above inequalities leads to
\begin{align*}
\norm{\nabla v}_{L^2(K_4)}^2 \geq Ch^2((c_2^s)^2 + (c_3^s)^2 + c_4^2h^2)
\end{align*}
where
\begin{align*}
C = \min\{12 - \frac{9}{\sigma}, 2(7-9\sigma)\} > 0.
\end{align*}
Then, \eqref{eq:bilinear_IFE_gradient_lower_bnd_1} and \eqref{eq:bilinear_IFE_gradient_lower_bnd_2}
follow by letting $\tK = K_4$.
\commentout{
Clear[dd, ee, tmp1, tmp2, tmp3, x, y]
Phi[x_, y_] = c1 + c2*x + c3*y + c4*x*y
Phidx[x_, y_] = D[Phi[x, y], {x, 1}]
Phidy[x_, y_] = D[Phi[x, y], {y, 1}]

"Integrals in the domain tK"
PhidxArea1 =
 Simplify[
  Integrate[Integrate[Phidx[x, y]^2, {y, h/2, h}], {x, h/2, h}]]
PhidyArea1 =
 Simplify[
  Integrate[Integrate[Phidy[x, y]^2, {y, h/2, h}], {x, h/2, h}]]
}
\end{proof}

Now, we are ready to establish the trace inequality for bilinear IFE functions on an interface element $K = \square A_1A_2A_3A_4$.
\begin{lemma}\label{lem:blinear_IFE_trace_ineq}
There exists a constant $C$ independent of the interface location such that for every bilinear IFE function $v(x,y)$ on $K$ the following
inequalities hold
\begin{align}
&\norm{\beta v_p}_{L^2(B)} \leq C h^{1/2} \abs{K}^{-1/2}\norm{\sqrt{\beta} \nabla v}_{L^2(K)}, ~~p = x,  y,  \label{eq:blinear_IFE_trace_ineq_1}\\
&\norm{\beta \nabla v \cdot \bfn_B}_{L^2(B)} \leq C h^{1/2} \abs{K}^{-1/2}\norm{\sqrt{\beta} \nabla v}_{L^2(K)}. \label{eq:blinear_IFE_trace_ineq_2}
\end{align}
\end{lemma}
\begin{proof}
 Without loss of generality, we assume that $K$ is a Type I interface element and
$B = \overline{A_1A_3}$ is an interface edge. To be more specific, we also assume that
$A_4 \in K^+$ and $\abs{K^+} \geq \frac{1}{2}\abs{K}$. Then
\begin{align*}
B = B^- \cup B^+, ~~B^- = \overline{A_1D}, ~B^+ = \overline{DA_3}.
\end{align*}
Direct calculations lead to
\begin{eqnarray}
  \qquad\quad\norm{\beta^- v_x}_{L^2(B^-)}^2 &=& (\beta^-)^2\big(dh(c_2^-)^2 + d^2h^2c_2^-c_4 + \frac{1}{3}d^3h^3 c_4^2\big),
\label{eq:trace_inq_bilinear_vx_B-} \\
  \norm{\beta^- v_y}_{L^2(B^-)}^2 &=& (\beta^-)^2dh (c_3^-)^2, \label{eq:trace_inq_bilinear_vy_B-} \\
  \norm{\beta^+ v_x}_{L^2(B^+)}^2 &=& \big(\beta^+\big)^2\left[(1-d)h(c_2^+)^2 + (1-d^2)h^2(c_2^+)c_4 + \frac{1}{3}(1-d^3)h^3 c_4^2\right], \label{eq:trace_inq_bilinear_vx_B+} \\
  \norm{\beta^+ v_y}_{L^2(B^+)}^2 &=& \big(\beta^+\big)^2(1-d)h(c_3^+)^2. \label{eq:trace_inq_bilinear_vy_B+}
\end{eqnarray}
%
Applying \eqref{eq:bilinear_IFE_gradient_lower_bnd_1} and \eqref{eq:bilinear_IFE_gradient_lower_bnd_2}
to \eqref{eq:trace_inq_bilinear_vx_B+} and \eqref{eq:trace_inq_bilinear_vy_B+} yields
\begin{equation}\label{eq:trace_inq_bilinear_vs_B-}
\quad\norm{\beta^+ v_p}_{L^2(B^+)}^2 \leq C \frac{h}{\abs{K}}\norm{\sqrt{\beta} \nabla v}_{L^2(K^+)}^2
\leq C \frac{h}{\abs{K}}\norm{\sqrt{\beta} \nabla v}_{L^2(K)}^2,~~p = x, y.
\end{equation}
Moreover, applying \eqref{eq:bilinear_IFE_coef_one_piece_bnd_by_another},
\eqref{eq:bilinear_IFE_gradient_lower_bnd_1}, and \eqref{eq:bilinear_IFE_gradient_lower_bnd_2}
to \eqref{eq:trace_inq_bilinear_vx_B-} and \eqref{eq:trace_inq_bilinear_vy_B-} leads to
\begin{equation}\label{eq:trace_inq_bilinear_vs_B+}
\norm{\beta^- v_p}_{L^2(B^-)}^2 \leq C \frac{h}{\abs{K}}\norm{\sqrt{\beta} \nabla v}_{L^2(K^+)}^2
\leq C \frac{h}{\abs{K}}\norm{\sqrt{\beta} \nabla v}_{L^2(K)}^2,~~p = x, y.
\end{equation}
Then, \eqref{eq:blinear_IFE_trace_ineq_1} follows from combining
\eqref{eq:trace_inq_bilinear_vs_B-} and \eqref{eq:trace_inq_bilinear_vs_B+} together.
Finally, \eqref{eq:blinear_IFE_trace_ineq_2} obviously follows from \eqref{eq:blinear_IFE_trace_ineq_1}.
\end{proof}


\section{Error Estimation for Partially Penalized IFE Methods}
We show that the IFE solution to the interface problem solved from \eqref{eq:IFE_eq} has an optimal convergence from
the point of the polynomials used in the involved IFE spaces. Unless otherwise specified, we always
assume that $\cT_h, 0<h<1$ is a family of regular Cartesian triangular or rectangular meshes \cite{PCiarlet_FEM}.

We start from proving the coercivity of the bilinear form $a_h(\cdot, \cdot)$ defined in
\eqref{eq:IFE_BF} on the IFE space $S_h(\Omega)$ with respect to the following energy norm:
\begin{equation} \label{eq:IP-IFE_method_dis_H1_norm}
\norm{v_h}_h = \left(\sum_{K \in \cT_h} \int_K \beta \nabla v_h \cdot \nabla v_h dX +  \sum_{B \in {\mathring \cE}_h^i}\int_B \frac{\sigma_B^0}{\abs{B}^\alpha} [v_h][v_h] ds\right)^{1/2}.
\end{equation}

\begin{lemma} \label{lem:coercivity}
There exists a constant $\kappa>0$ such that
\begin{equation}
\kappa \norm{v_h}_h^2 \leq a_h(v_h, v_h),~~\forall v_h \in S_h(\Omega) \label{eq:coercivity}
\end{equation}
is true for $\epsilon = 1$ unconditionally and is true for $\epsilon = 0$ or $\epsilon = -1$ under the condition that the stabilization parameter $\sigma_B^0$ in $a_h(\cdot, \cdot)$ is large enough. \\
\end{lemma}
\begin{proof}
 First, for $\epsilon = 1$, we note that the coercivity follows directly from the definitions of
$a_h(\cdot, \cdot)$  and $\norm{\cdot}_h$.

For $\epsilon = -1, 0$, note that
\begin{align}
a_h(v_h, v_h) &= \sum_{K \in \cT_h} \int_K \beta \nabla v_h \cdot \nabla v_h dX + (\epsilon - 1) \sum_{B \in {\mathring \cE}_h^i} \int_B \left\{\beta \nabla v_h \cdot \bfn_B\right\}[v_h] ds \label{eq:coercivity_2} \\
& + \sum_{B \in {\mathring \cE}_h^i}\int_B \frac{\sigma_B^0}{\abs{B}^\alpha} [v_h][v_h] ds, \nonumber
\end{align}
and the main concern is the second term on the right hand side. For each interface
edge $B \in {\mathring \cE}_h^i$ we let $K_{B,i}\in \cT_h, i = 1, 2$ be the two elements sharing $B$ as their common edge. Then, by the trace inequality \eqref{eq:linear_IFE_trace_ineq_2} or \eqref{eq:blinear_IFE_trace_ineq_2} and using $\alpha \geq 1$, we have,
\begin{eqnarray*}
   & & \int_B \left\{\beta \nabla v_h \cdot \bfn_B\right\}[v_h] ds \leq \norm{\left\{\beta \nabla v_h \cdot \bfn_B\right\}}_{L^2(B)} \norm{[v_h]}_{L^2(B)} \\
   &\leq& \left(\frac{1}{2}\norm{(\beta \nabla v_h \cdot \bfn_B)|_{K_{B,1}}}_{L^2(B)}
+\frac{1}{2}\norm{(\beta \nabla v_h \cdot \bfn_B)|_{K_{B,2}}}_{L^2(B)}\right)\norm{[v_h]}_{L^2(B)} \\
   &\leq& \left(\frac{C}{2}h_{K_{B,1}}^{-1/2} \norm{\sqrt{\beta} \nabla v_h}_{L^2(K_{B,1})} +  \frac{C}{2}h_{K_{B,2}}^{-1/2} \norm{\sqrt{\beta} \nabla v_h}_{L^2(K_{B,2})}\right)\norm{[v_h]}_{L^2(B)} \\
   &=& \Frac{C}{2}\abs{B}^{\alpha/2}\left(h_{K_{B,1}}^{-1/2} \norm{\sqrt{\beta} \nabla v_h}_{L^2(K_{B,1})} +  h_{K_{B,2}}^{-1/2} \norm{\sqrt{\beta}\nabla v_h}_{L^2(K_{B,2})}\right)\frac{1}{\abs{B}^{\alpha/2}}\norm{[v_h]}_{L^2(B)} \\
   &\leq& C\left(\norm{\sqrt{\beta} \nabla v_h}_{L^2(K_{B,1})}^2 + \norm{\sqrt{\beta} \nabla v_h}_{L^2(K_{B,2})}^2\right)^{1/2} \frac{1}{\abs{B}^{\alpha/2}}\norm{[v_h]}_{L^2(B)}.
\end{eqnarray*}
%
%
Therefore, for any $\delta > 0$, we have
\begin{align}
& \sum_{B \in {\mathring \cE}_h^i} \int_B \left\{\beta \nabla v_h \cdot \bfn_B\right\}[v_h] ds \label{eq:coercivity_3} \\
\leq &\sum_{B \in {\mathring \cE}_h^i} C\left(\norm{\sqrt{\beta} \nabla v_h}_{L^2(K_{B,1})}^2 + \norm{\sqrt{\beta} \nabla v_h}_{L^2(K_{B,2})}^2\right)^{1/2} \frac{1}{\abs{B}^{\alpha/2}}\norm{[v_h]}_{L^2(B)} \nonumber \\
\leq & \frac{\delta}{2}\sum_{K\in \cT_h}\norm{\sqrt{\beta} \nabla v_h}_{L^2(K)}^2 + \frac{C}{2\delta}\sum_{B \in {\mathring \cE}_h^i}\frac{1}{\abs{B}^{\alpha}}\norm{[v_h]}_{L^2(B)}^2.\nonumber
\end{align}
Then for $\epsilon = 0$ we let $\delta = 1$ and $\sigma_B^0 = C$, and for $\epsilon = -1$ we let $\delta = 1/2$ and $\sigma_B^0 = 5C/2$, where $C$ is in the above inequality.
The coercivity result \eqref{eq:coercivity} follows from using these parameters in \eqref{eq:coercivity_3} and
putting it in \eqref{eq:coercivity_2}.
\end{proof}

In the error estimation for the IFE solution, we need to use the fact that both linear and bilinear
IFE spaces have the optimal approximation capability
\cite{XHe_Thesis_Bilinear_IFE,XHe_TLin_YLin_Bilinear_Approximation,ZLi_TLin_YLin_RRogers_linear_IFE}. In particular, for every
$u \in \tH_0^2(\Omega)$ satisfying the interface jump conditions \eqref{eq:bvp_int_1} and \eqref{eq:bvp_int_2},
there exists a constant $C$ such that the interpolation $I_hu$ in the (either linear or bilinear) IFE space
$S_h(\Omega)$ has the following error bound:
\begin{align}
\norm{u - I_hu}_{L^2(\Omega)} + h\left(\sum_{T\in\mathcal{T}_h}\norm{u - I_hu}^2_{H^1(T)}\right)^{\frac{1}{2}} \leq Ch^2\norm{u}_{\tilde H^2(\Omega)}. \label{eq:intp_error_bnd}
\end{align}
In addition, we also need the error bound for $I_hu$ on interface edges which is given in the following
lemma.

\begin{lemma}\label{lem:interp_error_bnd_edge}
For every
$u \in \tH^{3}(\Omega)$ satisfying the interface jump conditions \eqref{eq:bvp_int_1} and \eqref{eq:bvp_int_2}, there exists a constant $C$ independent of the
interface such that its interpolation $I_hu$ in the IFE space
$S_h(\Omega)$ has the following error bound:
\begin{align}\label{eq:interp_error_bnd_edge}
\norm{\beta (\nabla(u-I_hu))|_{K} \cdot \bfn_B}_{L^2(B)}^2 \leq
C\big(h^2\norm{u}_{\tilde H^3(\Omega)}^2 + h \norm{u}_{\tilde H^2(K)}^2\big)
\end{align}
where $K$ is an interface element and $B$ is one of its interface edge.
\end{lemma}
\begin{proof}
We give a proof for linear IFEs, and the arguments can be used to establish this error bound for
bilinear IFEs.

Without loss of generality, let $K = \bigtriangleup A_1A_2A_3$ be an interface triangle such that
\begin{align}\label{eq: triangle vertices}
A_1 = (0,h), A_2 = (0,0), A_3 = (h,0)
\end{align}
and assume that the interface points on the edge of $K$ are
\begin{align}\label{eq: triangle DE}
D = (0,d), E = (e,h-e)
\end{align}
with $A_1 \in K^+$. Also we only discuss $B = \overline{A_1A_2}$, the estimate on the other
interface edge can be established similarly.

By Lemma 3.3 and Lemma 3.4 in \cite{ZLi_TLin_YLin_RRogers_linear_IFE}, for every $X \in \overline{DA_2}$, we have
\begin{align}
(I_{h}u(X) - u(X))_p &= \big(N^-(D) - N_{\overline{DE}}\big)\nabla u^-(X)(A_1-D) \pderiv{\phi_1(X)}{p}
\nonumber \\
&  +I_1(X) \pderiv{\phi_1(X)}{p} + I_2(X)\pderiv{\phi_2(X)}{p} + I_3(X) \pderiv{\phi_3(X)}{p},~~p = x, y,
\label{eq:interp_error_bnd_edge_1}
\end{align}
where
\begin{align*}
N^-(D) = \begin{pmatrix}
n_y(D)^2 + \rho n_x(D)^2 & (\rho - 1)n_x(D)n_y(D) \\
(\rho - 1)n_x(D)n_y(D) & n_y(D)^2 + \rho n_x(D)^2
\end{pmatrix},N_{\overline{DE}} = \begin{pmatrix}
{\bar n}_y^2 + \rho {\bar n}_x^2 & (\rho - 1){\bar n}_x {\bar n}_y \\
(\rho - 1){\bar n}_x {\bar n}_y & {\bar n}_y^2 + \rho {\bar n}_x^2
\end{pmatrix}
\end{align*}
$\rho = \beta^- /\beta^+$, ${\bf n}(X) = (n_x(X), n_y(X))^T$ is the normal to $\Gamma$ at $X$,
${\bf n}(\overline{DE}) = ({\bar n}_x, {\bar n}_y)^T$ is the normal of $\overline{DE}$,
and
\begin{align}\label{eq: I1}
  I_1(X) & = (1-t_d)\big(N^-(D) - I)\int_0^1 \oderiv{\nabla u^-}{t}(tD + (1-t)X)\cdot(A_1-X) dt  \nonumber\\
   & ~+ \int_0^{t_d}(1-t) \oderivm{u^-}{t}{2}(tA_1 + (1-t)X) dt + \int_{t_d}^1(1-t) \oderivm{u^+}{t}{2}(tA_1 + (1-t)X) dt,
\end{align}
\begin{align}\label{eq: I2I3}
  I_i(X) &= \int_0^1(1-t) \oderivm{u^-}{t}{2}(tA_2 + (1-t)X) dt, ~~~~~i = 2,3,
\end{align}
where $D = t_dA_1 + (1-t_d)X = X + t_d(A_1-X)$.
By Lemma 3.1 and Theorem 2.4 of \cite{ZLi_TLin_YLin_RRogers_linear_IFE}, we have
\begin{align}
& \int_{\overline{DA_2}}\left(\big(N^-(D) - N_{\overline{DE}}\big)\nabla u^-(X)(A_1-D) \pderiv{\phi_1(X)}{p}\right)^2dX \leq C h^3 \norm{u}_{H^3(\Omega^-)}^2,
\label{eq:interp_error_bnd_edge_2}
\end{align}
for $p=x,y$. By direct calculations we have
\begin{align*}
&\abs{\oderiv{\nabla u^-}{t}(tD + (1-t)X)\cdot(A_1-X)} \leq \big(\abs{u_{xx}^-(tD + (1-t)X)(x_d - x)(x_1-x)} \\
& \quad + \abs{u_{xy}^-(tD + (1-t)X)(y_d - y)(x_1-x)}  + \abs{u_{yx}^-(tD + (1-t)X)(x_d - x)(y_1-y)} \\
& \quad+ \abs{u_{yy}^-(tD + (1-t)X)(y_d-y)(y_1-y)}\big),
\end{align*}
and
\begin{align*}
&\abs{\oderivm{u^s}{t}{2}(tA_i + (1-t)X)} \\
\leq & \big(\abs{u_{xx}^s(tA_i + (1-t)X)(x_i - x)(x_i-x)}
+ \abs{u_{xy}^s(tA_i + (1-t)X)(y_i - y)(x_i-x)} \\
& ~~+ \abs{u_{yx}^s(tA_i + (1-t)X)(x_i-x)(y_i-y)}
+ \abs{u_{yy}^s(tA_i + (1-t)X)(y_i - y)(y_i-y)}\big)
\end{align*}
where $s = \pm, i = 1, 2, 3$. Let $I_{1,i}(X), i = 1, 2, 3$ be three integrals in $I_1(X)$, respectively. Then,
by Theorem 2.4 of \cite{ZLi_TLin_YLin_RRogers_linear_IFE}, we have
\begin{eqnarray}
    \int_{\overline{DA_2}} \left(I_{1,1}(X)\pderiv{\phi_1(X)}{p}\right)^2 dX
   &\leq&  \frac{C}{h^2} \int_0^d (1-t_d)^2 \int_0^1 \abs{u_{yy}^-(0, ty_d + (1-t)y)(y_d - y)(h-y)}^2 dt dy  \nonumber \\
   &\leq&  Ch^2 \int_0^d\abs{u_{yy}^-(0,z)}^2 dz \leq C h^2 \norm{u}_{H^3(\Omega^-)}^2,\label{eq:interp_error_bnd_edge_3}
\end{eqnarray}
\begin{eqnarray}
    \int_{\overline{DA_2}} \left(I_{1,2}(X)\pderiv{\phi_1(X)}{p}\right)^2 dX
   &\leq&  C \int_0^d \int_0^{t_d}\abs{u_{yy}^-(0,y + t(h-y))}^2(h-y)^2(1-t)^2 dt dy  \nonumber \\
   &\leq&  Ch^2 \int_0^d\abs{u_{yy}^-(0,z)}^2 dz \leq C h^2 \norm{u}_{H^3(\Omega^-)}^2,\label{eq:interp_error_bnd_edge_4}
\end{eqnarray}
\begin{eqnarray}
   \int_{\overline{DA_2}} \left(I_{1,3}(X)\pderiv{\phi_1(X)}{p}\right)^2 dX &\leq& C \int_0^d \int_{t_d}^1\abs{u_{yy}^+(0,y + t(h-y))}^2(h-y)^2(1-t)^2 dt dy  \nonumber \\
   &\leq& C h^2 \int_{d}^h\abs{u_{yy}^+(0,z)}^2 dz \leq C h^2 \norm{u}_{H^3(\Omega^+)}^2.\label{eq:interp_error_bnd_edge_5}
\end{eqnarray}
Similarly, we can show that
\begin{eqnarray}
  \int_{\overline{DA_2}} \left(I_2(X)\pderiv{\phi_2(X)}{p}\right)^2 dX
  &\leq& C \int_0^d \int_0^1 \abs{u_{yy}^-(0,(1-t)y)}^2 y^2(1-t)^2 dt dy \nonumber \\
   &\leq& Ch^2 \int_0^d \abs{u_{yy}^-(0,z)}^2dz \leq Ch^2 \norm{u}_{H^3(\Omega^-)}^2. \label{eq:interp_error_bnd_edge_6}
\end{eqnarray}
For the term involving $I_3(X)$, we have
\begin{eqnarray}\label{eq:interp_error_bnd_edge_7}
   && \qquad\qquad\qquad\int_{\overline{DA_2}} \left(I_3(X)\pderiv{\phi_3(X)}{p}\right)^2 dX\\
   &\leq& C \left(h^2\int_0^d \int_0^1 \abs{u_{xx}^-(th,(1-t)y)}^2 (1-t)^2 dt dy
+ \int_0^d \int_0^1 \abs{u_{xy}^-(th,(1-t)y)}^2y^2(1-t)^2 dt dy \right. \nonumber \\
   &+&  \left.\int_0^d \int_0^1 \abs{u_{yx}^-(th,(1-t)y)}^2y^2(1-t)^2 dt dy + \int_0^d \int_0^1 \abs{u_{yy}^-(th,(1-t)y)}^2 y^2(1-t)^2 dt dy \right).\nonumber
\end{eqnarray}
Let $th = p, (1-t)y = q$, then we have
\begin{align*}
t = \frac{p}{h}, y = \frac{q}{1-t} = \frac{q}{1- p/h},
\end{align*}
and,
\begin{align*}
&\frac{q^2}{h-p} = \frac{(1-t)^2y^2}{h - th} = (1-t)y\frac{y}{h}, \abs{(1-t)y} \leq h,  \frac{y}{h} \leq 1.
\end{align*}
Hence,
\begin{align*}
&h^2\int_0^d \int_0^1 \abs{u_{xx}^-(th,(1-t)y)}^2 (1-t)^2 dt dy = \iint \limits_{\bigtriangleup_{DA_2A_3}} \abs{u_{xx}^-(p,q)}^2 \frac{h^2(1-p/h)^2}{h-p} dp dq \\
& \leq \iint \limits_{\bigtriangleup_{DA_2A_3}} \abs{u_{xx}^-(p,q)}^2 (h-p) dp dq \leq C h \norm{u}_{\tilde H^2(K)}^2, \\
&\int_0^d \int_0^1 \abs{u_{xy}^-(th,(1-t)y)}^2 y^2(1-t)^2 dt dy = \iint \limits_{\bigtriangleup_{DA_2A_3}} \abs{u_{xy}^-(p,q)}^2 \frac{q^2}{h-p} dp dq
\leq C h \norm{u}_{\tilde H^2(K)}^2, \\
&\int_0^d \int_0^1 \abs{u_{yy}^-(th,(1-t)y)}^2 y^2(1-t)^2 dt dy =\iint \limits_{\bigtriangleup_{DA_2A_3}} \abs{u_{yy}^-(p,q)}^2 \frac{q^2}{h-p} dp dq
\leq C h \norm{u}_{\tilde H^2(K)}^2.
\end{align*}
Using these estimates in \eqref{eq:interp_error_bnd_edge_6}, we have
\begin{align}
&\int_{\overline{DA_2}} \left(I_3(X)\pderiv{\phi_3(X)}{p}\right)^2 dX \leq C h \norm{u}_{\tilde H^2(K)}^2.
\label{eq:interp_error_bnd_edge_8}
\end{align}
Finally, the inequality \eqref{eq:interp_error_bnd_edge} follows from putting estimates
\eqref{eq:interp_error_bnd_edge_2}-\eqref{eq:interp_error_bnd_edge_6}, \eqref{eq:interp_error_bnd_edge_7}
into \eqref{eq:interp_error_bnd_edge_1}.
\end{proof}

\begin{remark}\label{lem:interp_error_bnd_edge_inf}
For every
$u \in \tW^{2,\infty}(\Omega)$ satisfying the interface jump conditions \eqref{eq:bvp_int_1} and \eqref{eq:bvp_int_2},
we can also show that there exists a constant $C$ independent of the interface such that the interpolation $I_hu$ in the IFE space
$S_h(\Omega)$ fulfills
\begin{align}\label{eq:interp_error_bnd_edge_inf}
\norm{\beta (\nabla(u-I_hu))|_{K} \cdot \bfn_B}_{L^2(B)}^2 \leq Ch^3\norm{u}_{\tilde W^{2,\infty}(\Omega)}^2,
\end{align}
where $K$ is an interface element and $B$ is one of its interface edge. A proof for
\eqref{eq:interp_error_bnd_edge_inf} is given in Appendix \ref{remark_1} which uses arguments similar to
those for proving Lemma \ref{lem:interp_error_bnd_edge}.
\end{remark}

Now, we are ready to derive the error bound for IFE solutions generated by the partially penalized IFE method
\eqref{eq:IFE_eq}.

\begin{theorem}\label{th:error_bnd_H1}
Assume that the exact solution $u$ to the interface problem \eqref{bvp_pde}-\eqref{eq:bvp_int_2} is
in $\tH^3(\Omega)$ and $u_h$ is its IFE solution generated with
$\alpha = 1$ on a Cartesian (either triangular or rectangular) mesh $\mathcal{T}_h$. Then there exists
a constant $C$ such that
\begin{align}
\norm{u - u_h}_h \leq C h \norm{u}_{\tilde H^3(\Omega)}. \label{eq:error_bnd_H1}
\end{align}
\end{theorem}
\begin{proof}
From the weak form \eqref{eq:weak_form} and the IFE equation \eqref{eq:IFE_eq} we have
\begin{align}
&a_h(v_h, u_h - w_h) = a_h(v_h, u-w_h), ~~\forall v_h, w_h \in S_h(\Omega). \label{eq:error_bnd_H1_1}
\end{align}
Letting $v_h = u_h - w_h$ in \eqref{eq:error_bnd_H1_1} and using the coercivity of $a_h(\cdot, \cdot)$, we have
\begin{eqnarray}
   && \qquad\qquad\kappa \norm{u_h - w_h}_h^2 \leq \abs{a_h(u_h - w_h, u_h - w_h)} = \abs{a_h(u_h - w_h, u-w_h)} \label{eq:error_bnd_H1_2}\\
   && \leq \left|\sum_{K \in \cT_h} \int_K \beta \nabla(u_h - w_h) \cdot \nabla(u - w_h) dX\right| +
\left|\sum_{B \in {\mathring \cE}_h^i} \int_B \left\{\beta \nabla(u-w_h) \cdot \bfn_B\right\}[u_h - w_h] ds \right| \nonumber\\
   && \left| \epsilon \sum_{B \in {\mathring \cE}_h^i} \int_B \left\{\beta \nabla(u_h-w_h) \cdot \bfn_B\right\}[u - w_h] ds\right|
+ \left|\sum_{B \in {\mathring \cE}_h^i}\int_B \frac{\sigma_B^0}{\abs{B}^\alpha} [u_h-w_h][u-w_h] ds \right|.
\nonumber
\end{eqnarray}
%
We denote the four terms on the right of \eqref{eq:error_bnd_H1_2} by $Q_i, i = 1, 2, 3, 4$. Then,
\begin{eqnarray}
  Q_1 &\leq& \left(\sum_{K \in \cT_h} \norm{\beta^{1/2} \nabla(u - w_h)}_{L^2(K)}^2\right)^{1/2}\left(\sum_{K \in \cT_h} \norm{\beta^{1/2} \nabla(u_h - w_h)}_{L^2(K)}^2\right)^{1/2} \nonumber \\
   &\leq& \frac{3}{2\kappa}\max(\beta^-, \beta^+) \norm{\nabla(u-w_h)}_{L^2(\Omega)}^2 + \frac{\kappa}{6}\sum_{K \in \cT_h} \norm{\beta^{1/2} \nabla(u_h - w_h)}_{L^2(K)}^2 \nonumber \\
   &\leq& C\norm{\nabla(u-w_h)}_{L^2(\Omega)}^2 + \frac{\kappa}{6}\norm{u_h - w_h}_h^2. \label{eq:error_bnd_H1_Q1}
\end{eqnarray}
\begin{eqnarray}
  Q_2 &\leq& \frac{\kappa}{6}\sum_{B \in {\mathring \cE}_h^i}\frac{\sigma_B^0}{\abs{B}^\alpha}\norm{[u_h - w_h]}_{L^2(B)}^2 + C\sum_{B \in {\mathring \cE}_h^i}\frac{\abs{B}^\alpha}{\sigma_B^0}\norm{\left\{\beta \nabla(u-w_h) \cdot \bfn_B\right\}}_{L^2(B)}^2 \nonumber \\
   &\leq& \frac{\kappa}{6}\norm{u_h - w_h}_h^2 + C\sum_{B \in {\mathring \cE}_h^i}\frac{\abs{B}^\alpha}{\sigma_B^0}\norm{\left\{\beta \nabla(u-w_h) \cdot \bfn_B\right\}}_{L^2(B)}^2. \label{eq:error_bnd_H1_Q2}
\end{eqnarray}
To bound $Q_3$, for each $B \in {\mathring \cE}_h^i
$, we let $K_{B,i} \in \cT_h, i = 1, 2$ be such that $B = K_{B,1}\cap K_{B,2}$. First,
by the standard trace inequality on elements for $H^1$ functions, we have
\begin{align*}
\norm{[u - w_h]}_{L^2(B)} &\leq \norm{(u-w_h)|_{K_{B,1}}}_{L^2(B)} + \norm{(u-w_h)|_{K_{B,2}}}_{L^2(B)} \\
&\leq Ch^{-1/2}\left(\norm{u - w_h}_{L^2(K_{B,1})} + h\norm{\nabla(u - w_h)}_{L^2(K_{B,1})}\right) \\
&\hspace{0.2in} + Ch^{-1/2}\left(\norm{u - w_h}_{L^2(K_{B,2})} + h\norm{\nabla(u - w_h)}_{L^2(K_{B,2})}\right). 
\end{align*}
Then, applying the trace inequalities established in
Lemma \ref{lem:linear_IFE_trace_ineq} or Lemma \ref{lem:blinear_IFE_trace_ineq} depending on
whether linear IFEs or bilinear IFE are considered, we have
\begin{eqnarray*}
   && \norm{\left\{\beta \nabla(u_h-w_h) \cdot \bfn_B\right\}}_{L^2(B)} \\
   &\leq& \frac{C}{2}h^{-1/2}\left(\norm{\sqrt{\beta} \nabla (u_h-w_h)}_{L^2(K_{B,1})} + \norm{\sqrt{\beta} \nabla (u_h-w_h)}_{L^2(K_{B,2})}\right).
\end{eqnarray*}
%
Hence
\begin{align}
Q_3 &\leq \abs{\epsilon} \sum_{B \in {\mathring \cE}_h^i} \norm{\left\{\beta \nabla(u_h-w_h) \cdot \bfn_B\right\}}_{L^2(B)} \norm{[u - w_h]}_{L^2(B)} \nonumber \\
&\leq Ch^{-2}\left(\norm{u - w_h}_{L^2(\Omega)}^2 + h^2\norm{\nabla(u - w_h)}_{L^2(\Omega)}^2  \right)
+ \frac{\kappa}{6}\norm{u_h-w_h}_h^2.
\label{eq:error_bnd_H1_Q3}
\end{align}
To bound $Q_4$, by the standard trace inequality, we have
\begin{align}
&\int_B \frac{\sigma_B^0}{\abs{B}^\alpha} [u-w_h][u-w_h] ds= \frac{\sigma_B^0}{\abs{B}^\alpha}\norm{[u-w_h]}_{L^2(B)}^2 \nonumber \\
&\leq  \frac{\sigma_B^0}{\abs{B}^\alpha} \left(\norm{(u-w_h)|_{K_{B,1}}}_{L^2(B)} + \norm{(u-w_h)|_{K_{B,2}}}_{L^2(B)}\right)^2 \nonumber \\
&\leq \frac{\sigma_B^0}{\abs{B}^\alpha}C\abs{B}\abs{K_{B,1}}^{-1}\big(\norm{u-w_h}_{L^2(K_{B,1})} +
h \norm{\nabla (u -w_h)}_{L^2(K_{B,1})}\big)^2 \nonumber \\
&\hspace{0.1in} + \frac{\sigma_B^0}{\abs{B}^\alpha}C\abs{B}\abs{K_{B,2}}^{-1}\big(\norm{u-w_h}_{L^2(K_{B,2})} +h \norm{\nabla (u -w_h)}_{L^2(K_{B,2})}\big)^2  \nonumber \\
& \leq Ch^{-(\alpha+1)}\big(\norm{u-w_h}_{L^2(K_{B,1})} +
h \norm{\nabla (u -w_h)}_{L^2(K_{B,1})}\big)^2 \nonumber \\
&\hspace{0.1in} + Ch^{-(\alpha+1)}\big(\norm{u-w_h}_{L^2(K_{B,2})} +h \norm{\nabla (u -w_h)}_{L^2(K_{B,2})}\big)^2, \label{eq:error_bnd_H1_7}
\end{align}
where we have used the facts that
\begin{align*}
&\abs{B} = h \text{~~or~~} \abs{B} = \sqrt{2}h, ~~\abs{K_{B,i}} = \Frac{h^2}{2}, \text{~~or~~} \abs{K_{B,i}} = h^2, i = 1, 2.
\end{align*}
Then
\begin{align}
Q_4 &\leq \sum_{B \in {\mathring \cE}_h^i}\left(\frac{\kappa}{6} \int_B \frac{\sigma_B^0}{\abs{B}^\alpha} [u_h-w_h][u_h-w_h] ds
+ \frac{3}{2\kappa} \int_B \frac{\sigma_B^0}{\abs{B}^\alpha} [u-w_h][u-w_h] ds \right) \nonumber \\
&\leq \frac{\kappa}{6}\norm{u_h-w_h}_h^2 + \frac{3}{2\kappa}\sum_{B \in {\mathring \cE}_h^i}\int_B \frac{\sigma_B^0}{\abs{B}^\alpha} [u-w_h][u-w_h] ds \nonumber \\
&\leq \frac{\kappa}{6}\norm{u_h-w_h}_h^2 + Ch^{-(\alpha+1)}\big(\norm{u-w_h}_{L^2(\Omega)}^2 +
h^2 \norm{\nabla (u -w_h)}_{L^2(\Omega)}^2\big). \label{eq:error_bnd_H1_Q4}
\end{align}
Then, we put all these bounds for $Q_i, i = 1, 2, 3, 4$ in \eqref{eq:error_bnd_H1_2} to have
\begin{eqnarray*}
  \norm{u_h - w_h}_h^2 &\leq& C\norm{\nabla(u-w_h)}_{L^2(\Omega)}^2 +  C\sum_{B \in {\mathring \cE}_h^i}\frac{\abs{B}^\alpha}{\sigma_B^0}\norm{\left\{\beta \nabla(u-w_h) \cdot \bfn_B\right\}}_{L^2(B)}^2 \\
   && +~ Ch^{-2}\left(\norm{u - w_h}_{L^2(\Omega)}^2 + h^2\norm{\nabla(u - w_h)}_{L^2(\Omega)}^2  \right)\\
   && +~  Ch^{-(\alpha+1)}\big(\norm{u-w_h}_{L^2(\Omega)}^2 +
h^2 \norm{\nabla (u -w_h)}_{L^2(\Omega)}^2\big).
\end{eqnarray*}
Hence, we let $w_h = I_hu$ in the above and using the optimal approximation capability of linear and bilinear
IFE spaces \eqref{eq:intp_error_bnd}
to have
\begin{align}
&\norm{u_h - I_hu}_h^2 \leq C\big(h^2+h^{4-(\alpha+1)}\big)\norm{u}_{\tilde H^2(\Omega)}^2 
+C\sum_{B \in {\mathring \cE}_h^i}\frac{\abs{B}^\alpha}{\sigma_B^0}\norm{\left\{\beta \nabla(u-I_hu) \cdot \bfn_B\right\}}_{L^2(B)}^2. \label{eq:error_bnd_H1_3}
\end{align}
For the second term on the right in \eqref{eq:error_bnd_H1_3}, we use
Lemma \ref{lem:interp_error_bnd_edge} to bound it:
\begin{align}
&\sum_{B \in {\mathring \cE}_h^i}\frac{\abs{B}^\alpha}{\sigma_B^0}\norm{\left\{\beta \nabla(u-I_hu) \cdot \bfn_B\right\}}_{L^2(B)}^2 \nonumber \\
\leq & \sum_{B \in {\mathring \cE}_h^i}\frac{\abs{B}^\alpha}{\sigma_B^0}
C\big(h^2 \norm{u}_{3,\Omega}^2 + h\norm{u}_{2, K_{B,1}} + h\norm{u}_{2, K_{B,1}}\big)
\leq Ch^{1+\alpha}\norm{u}_{\tilde H^3(\Omega)}^2.
\label{eq:error_bnd_H1_4}
\end{align}
Here we have used the fact that the number of interface elements is of $O(h^{-1})$. Hence, for
$\alpha = 1$, we can combine \eqref{eq:error_bnd_H1_3} and \eqref{eq:error_bnd_H1_4} to have
\begin{align}
&\norm{u_h - I_hu}_h \leq Ch\norm{u}_{\tilde H^3(\Omega)}.
\label{eq:error_bnd_H1_5}
\end{align}
In addition, using the optimal approximation capability of linear and
bilinear IFE spaces \eqref{eq:intp_error_bnd}
and \eqref{eq:error_bnd_H1_7}, we can show that
\begin{align}
\norm{u - I_hu}_h\leq Ch\norm{u}_{\tilde H^2(\Omega)}.
\label{eq:error_bnd_H1_6}
\end{align}
Finally, the error estimate \eqref{eq:error_bnd_H1} follows from applying
\eqref{eq:error_bnd_H1_5} and \eqref{eq:error_bnd_H1_6} to the following standard inequality:
\begin{equation*}
\norm{u - u_h}_h \leq \norm{u-I_hu}_h + \norm{u_h - I_hu}_h.
\end{equation*}
\end{proof}

\begin{remark}\label{th:error_bnd_H1_inf}
If the exact solution $u$ to the interface problem \eqref{bvp_pde}-\eqref{eq:bvp_int_2} is
in the function space $\tW^{2,\infty}(\Omega)$, then, using Remark \ref{lem:interp_error_bnd_edge_inf} and arguments similar to those for the proof of Theorem \ref{th:error_bnd_H1}, we can show that
the IFE solution $u_h$ generated with
$\alpha = 1$ on a Cartesian (either triangular or rectangular) mesh $\mathcal{T}_h$ has the following
error estimate:
\begin{align}
\norm{u - u_h}_h \leq C\big(h \norm{u}_{\tilde H^2(\Omega)} + h^{3/2}\norm{u}_{\tW^{2,\infty}(\Omega)}\big). \label{eq:error_bnd_H1_inf}
\end{align}
\end{remark}

\section{Numerical Examples}

In this section, we present a couple of numerical examples to demonstrate features of the
partially penalized IFE methods for elliptic interface problems.

For comparison, the interface problem to be solved in the numerical examples is the same as the one in \cite{XHe_TLin_YLin_Bilinear_Approximation}. Specifically, we consider the interface problem \eqref{bvp_pde}-\eqref{eq:bvp_int_2}
except that \eqref{eq:bvp_bc} is replaced by the non-homogeneous boundary condition $u|_{\partial \Omega} = g$. Let the solution domain $\Omega$ be the open rectangle $(-1,1)\times(-1,1)$ and the interface $\Gamma$ be the circle centered at origin point with a radius $r_0 = \pi/6.28$ which separates $\Omega$ into two sub-domains, denoted by $\Omega^-$ and $\Omega^+$, \emph{i.e.},
\begin{equation*}
    \Omega^- = \{(x,y): x^2+y^2 < r_0^2\}, ~~~\text{and}~~~
    \Omega^+ = \{(x,y): x^2+y^2 > r_0^2\}.
\end{equation*}
The exact solution $u$ to the interface problem is chosen as follows
\begin{equation}\label{eq: true solution}
    u(x,y) =
    \left\{
      \begin{array}{ll}
        \frac{r^\alpha}{\beta^-}, & \text{if~} r \leq r_0, \\
        \frac{r^\alpha}{\beta^+} + \left(\frac{1}{\beta^-} - \frac{1}{\beta^+}\right), & \text{otherwise},
      \end{array}
    \right.
\end{equation}
where $\alpha = 5$, and $r = \sqrt{x^2+y^2}$. The functions $f$ and $g$ in this interface problem are consequently
determined by $u$. The Cartesian meshes $\mathcal{T}_h, h>0$ are formed by
partitioning $\Omega$ into $N\times N$ congruent squares of size $h = 2/N$ for a set of values of integer $N$.

To describe our numerical results, we rewrite the bilinear form in the partially penalized IFE methods \eqref{eq:IFE_eq} as
follows
\begin{eqnarray}\label{eq: bilinear form}
    a_{h} (u_h, v_h) &=& \sum_{T\in \mathcal{T}_h} \int_T\beta \nabla u_h \cdot \nabla v_h dX + \delta
    \sum_{B \in {\mathring \cE}_h^i} \int_B \{\beta\nabla u_h\cdot\mathbf{n}_B\}[v_h]ds  \nonumber \\
    &&~~~~ + \epsilon
    \sum_{B \in {\mathring \cE}_h^i} \int_B \{\beta\nabla v_h\cdot\mathbf{n}_B\}[u_h]ds  +
\sum_{B \in {\mathring \cE}_h^i} \frac{\sigma_B^0}{|B|}\int_{B}[u_h][v_h]ds.
\end{eqnarray}
When $\delta = 0$, $\epsilon = 0$, and $\sigma^0_B = 0$ for all $B\in {\mathring \cE}_h^i$, the partially penalized bilinear IFE method becomes the \emph{classic} bilinear IFE method proposed in \cite{XHe_Thesis_Bilinear_IFE, XHe_TLin_YLin_Bilinear_Approximation}. When $\delta = -1$, and $\sigma^0_B > 0$, we call the partially penalized IFE methods corresponding to $\epsilon = -1, 0, 1$, respectively, the symmetric, incomplete, and nonsymmetric PPIFE methods because of their similarity to the corresponding DG methods \cite{Riviere_DG_book}.
In our numerical experiment, we use the parameter $\alpha = 1$ as suggested by
our error estimation. Also, we use $\sigma_B^0 = 10\max(\beta^-,\beta^+)$ in the SPPIFE and IPPIFE schemes and $\sigma_B^0 = 1$ in the NPPIFE scheme.

In the first example, we test these IFE methods with the above interface problem whose diffusion coefficient has a large jump $(\beta^-,\beta^+) = (1,10000)$. Because of the large value of $\beta^+$, the exact solution $u(x,y)$ varies little in $\Omega^+$ which might be one of the reasons
this example is not overly difficult for all the IFE methods, and they indeed perform comparably well, see Tables \ref{table: IFE solution H1 error 1 10000}, \ref{table: IFE solution L2 error 1 10000}, and \ref{table: IFE solution infinity error 1 10000}. Specifically, all the partially penalized IFE methods converge optimally in the $H^1$-norm as predicted by the
error analysis in the previous section. The classic IFE method also converges optimally in the $H^1$-norm even though the related error bound has not been rigorously established yet. The data in Table \ref{table: IFE solution L2 error 1 10000} demonstrate that all the IFE methods converge in the $L^2$-norm at the expected optimal rate. Even though they all seem to converge in the $L^\infty$-norm, the data
in Table \ref{table: IFE solution infinity error 1 10000} do not reveal any definite rates for them.

\begin{table}[th]
\caption{Comparison of $|u_h - u|_{H^1(\Omega)}$ for different IFE methods with $\beta^- = 1$, $\beta^+ = 10000$.}
\vspace{-5mm}
\begin{center}
\begin{footnotesize}
\begin{tabular}{|c|cc|cc|cc|cc|}
\hline
& \multicolumn{2}{c|}{Classic IFE}
& \multicolumn{2}{c|} {SPP IFE}
& \multicolumn{2}{c|} {IPP IFE}
& \multicolumn{2}{c|} {NPP IFE}\\
\hline
$N$
& $|\cdot|_{H^1}$ & rate
& $|\cdot|_{H^1}$ & rate
& $|\cdot|_{H^1}$ & rate
& $|\cdot|_{H^1}$ & rate  \\
 \hline
$20$      &$1.9495E{-2}$ &         &$1.9538E{-2}$ &         &$1.9538E{-2}$ &         &$1.9482E{-2}$ &           \\
$40$      &$1.0539E{-2}$ & 0.8873  &$1.0601E{-2}$ & 0.8821  &$1.0602E{-2}$ & 0.8820  &$1.0562E{-2}$ & 0.8832    \\
$80$      &$5.4219E{-3}$ & 0.9590  &$5.5838E{-3}$ & 0.9249  &$5.5844E{-3}$ & 0.9248  &$5.4945E{-3}$ & 0.9428    \\
$160$     &$2.7157E{-3}$ & 0.9975  &$2.7720E{-3}$ & 1.0103  &$2.7724E{-3}$ & 1.0103  &$2.7291E{-3}$ & 1.0096    \\
$320$     &$1.3535E{-3}$ & 1.0046  &$1.3887E{-3}$ & 0.9973  &$1.3890E{-3}$ & 0.9970  &$1.3613E{-3}$ & 1.0035    \\
$640$     &$6.7876E{-4}$ & 0.9957  &$6.9915E{-4}$ & 0.9900  &$6.9923E{-4}$ & 0.9902  &$6.7915E{-4}$ & 1.0032    \\
$1280$    &$3.4023E{-4}$ & 0.9964  &$3.4439E{-4}$ & 1.0216  &$3.4447E{-4}$ & 1.0214  &$3.3979E{-4}$ & 0.9991    \\
$2560$    &$1.7079E{-4}$ & 0.9942  &$1.7036E{-4}$ & 1.0155  &$1.7038E{-4}$ & 1.0156  &$1.6994E{-4}$ & 0.9996    \\
\hline
\end{tabular}
\end{footnotesize}
\end{center}
\label{table: IFE solution H1 error 1 10000}
\end{table}

\begin{table}[th]
\caption{Comparison of $\|u_h - u\|_{L^2(\Omega)}$ for different IFE methods with $\beta^- = 1$, $\beta^+ = 10000$.}
\vspace{-5mm}
\begin{center}
\begin{footnotesize}
\begin{tabular}{|c|cc|cc|cc|cc|}
\hline
& \multicolumn{2}{c|}{Classic IFE}
& \multicolumn{2}{c|} {SPP IFE}
& \multicolumn{2}{c|} {IPP IFE}
& \multicolumn{2}{c|} {NPP IFE}\\
\hline
$N$
& $\|\cdot\|_{L^2}$ & rate
& $\|\cdot\|_{L^2}$ & rate
& $\|\cdot\|_{L^2}$ & rate
& $\|\cdot\|_{L^2}$ & rate  \\
 \hline
$20$      &$1.1175E{-3}$ &         &$1.1273E{-3}$ &         &$1.1276E{-3}$ &         &$1.1179E{-3}$ &        \\
$40$      &$2.8572E{-4}$ & 1.9676  &$2.9171E{-4}$ & 1.9503  &$2.9181E{-4}$ & 1.9501  &$2.8567E{-4}$ & 1.9683 \\
$80$      &$7.5990E{-5}$ & 1.9107  &$8.5150E{-5}$ & 1.7764  &$8.5223E{-5}$ & 1.7757  &$7.6342E{-5}$ & 1.9038 \\
$160$     &$1.8116E{-5}$ & 2.0685  &$2.2589E{-5}$ & 1.9144  &$2.2645E{-5}$ & 1.9120  &$1.8098E{-5}$ & 2.0766 \\
$320$     &$4.4753E{-6}$ & 2.0172  &$6.1332E{-6}$ & 1.8809  &$6.1629E{-6}$ & 1.8775  &$4.5193E{-6}$ & 2.0017 \\
$640$     &$1.0969E{-6}$ & 2.0286  &$1.6502E{-6}$ & 1.8940  &$1.6502E{-6}$ & 1.9010  &$1.1235E{-6}$ & 2.0081 \\
$1280$    &$2.6689E{-7}$ & 2.0391  &$3.7104E{-7}$ & 2.1530  &$3.7275E{-7}$ & 2.1464  &$2.7680E{-7}$ & 2.0211 \\
$2560$    &$6.3940E{-8}$ & 2.0615  &$7.3251E{-8}$ & 2.3407  &$7.3657E{-8}$ & 2.3393  &$6.8809E{-8}$ & 2.0082 \\
\hline
\end{tabular}
\end{footnotesize}
\end{center}
\label{table: IFE solution L2 error 1 10000}
\end{table}

\begin{table}[h]
\caption{Comparison of $\|u_h - u\|_{L^\infty(\Omega)}$ for different IFE methods with $\beta^- = 1$, $\beta^+ = 10000$.}
\vspace{-5mm}
\begin{center}
\begin{footnotesize}
\begin{tabular}{|c|cc|cc|cc|cc|}
\hline
& \multicolumn{2}{c|}{Classic IFE}
& \multicolumn{2}{c|} {SPP IFE}
& \multicolumn{2}{c|} {IPP IFE}
& \multicolumn{2}{c|} {NPP IFE}\\
\hline
$N$
& $\|\cdot\|_{L^\infty}$ & rate
& $\|\cdot\|_{L^\infty}$ & rate
& $\|\cdot\|_{L^\infty}$ & rate
& $\|\cdot\|_{L^\infty}$ & rate  \\
 \hline
$20$      &$8.8830E{-4}$ &         &$9.9803E{-4}$ &         &$9.9820E{-4}$ &         &$8.4663E{-4}$ &        \\
$40$      &$4.3525E{-4}$ & 1.0292  &$5.2256E{-4}$ & 0.9335  &$5.2185E{-4}$ & 0.9357  &$4.2059E{-4}$ & 1.0093 \\
$80$      &$1.6536E{-4}$ & 1.3962  &$2.2692E{-4}$ & 1.2034  &$2.2745E{-4}$ & 1.1981  &$1.6816E{-4}$ & 1.3225 \\
$160$     &$7.4603E{-5}$ & 1.1483  &$9.8441E{-5}$ & 1.2049  &$9.8619E{-5}$ & 1.2056  &$7.2067E{-5}$ & 1.2225 \\
$320$     &$1.2972E{-5}$ & 2.5238  &$3.6045E{-5}$ & 1.4494  &$3.6421E{-5}$ & 1.4371  &$1.4362E{-5}$ & 2.3271 \\
$640$     &$6.1332E{-6}$ & 1.0807  &$1.7607E{-5}$ & 1.0337  &$1.7531E{-5}$ & 1.0548  &$6.4702E{-5}$ & 1.1504 \\
$1280$    &$2.5136E{-6}$ & 1.2869  &$5.1961E{-6}$ & 1.7606  &$5.1799E{-6}$ & 1.7589  &$1.6993E{-6}$ & 1.9289 \\
$2560$    &$1.2810E{-6}$ & 0.9725  &$1.0242E{-6}$ & 2.3430  &$1.0274E{-6}$ & 2.3339  &$9.1534E{-7}$ & 0.8926 \\
\hline
\end{tabular}
\end{footnotesize}
\end{center}
\label{table: IFE solution infinity error 1 10000}
\end{table}

However, we have observed that the partially penalized IFE methods outperform the classic IFE method in many situations. We demonstrate this by numerical results generated by these IFE methods for the interface problem whose diffusion coefficient has a typical moderate jump $(\beta^-,\beta^+) = (1,10)$. IFE solution errors are listed in Tables \ref{table: IFE solution H1 error 1 10}, \ref{table: IFE solution L2 error 1 10}, and \ref{table: IFE solution infinity error 1 10}. From the data in Table \ref{table: IFE solution H1 error 1 10}, we can see that all the partially penalized IFE methods maintain their predicted $O(h)$ convergence rate
in the $H^1$-norm over all the meshes up to the finest one, while the classic IFE method slightly looses its convergence rate in the $H^1$-norm when the mesh becomes very fine. The effects of the penalization are more prominent when the errors are gauged in the $L^2$-norm and $L^\infty$-norm. According to the data
in Table \ref{table: IFE solution L2 error 1 10}, all the partially penalized IFE methods converge at the optimal rate in the $L^2$-norm but the $L^2$-norm rate of
the classic IFE method clearly degenerates when the mesh becomes finer. Similar phenomenon for the $L^\infty$-norm convergence can be observed from those data
in Table \ref{table: IFE solution infinity error 1 10}.

\begin{table}[ht]
\caption{Comparison of $|u_h - u|_{H^1(\Omega)}$ for different IFE methods with $\beta^- = 1$, $\beta^+ = 10$.}
\vspace{-5mm}
\begin{center}
\begin{footnotesize}
\begin{tabular}{|c|cc|cc|cc|cc|}
\hline
& \multicolumn{2}{c|}{Classic IFE}
& \multicolumn{2}{c|} {SPP IFE}
& \multicolumn{2}{c|} {IPP IFE}
& \multicolumn{2}{c|} {NPP IFE}\\
\hline
$N$
& $|\cdot|_{H^1}$ & rate
& $|\cdot|_{H^1}$ & rate
& $|\cdot|_{H^1}$ & rate
& $|\cdot|_{H^1}$ & rate  \\
 \hline
$20$      &$6.4758E{-2}$ &         &$6.4751E{-2}$ &         &$6.4753E{-2}$ &         &$6.4719E{-2}$ &         \\
$40$      &$3.2779E{-2}$ & 0.9823  &$3.2650E{-2}$ & 0.9878  &$3.2651E{-2}$ & 0.9878  &$3.2637E{-2}$ & 0.9877  \\
$80$      &$1.6596E{-2}$ & 0.9723  &$1.6386E{-2}$ & 0.9947  &$1.6386E{-2}$ & 0.9947  &$1.6382E{-2}$ & 0.9944  \\
$160$     &$8.4072E{-3}$ & 0.9811  &$8.2081E{-3}$ & 0.9974  &$8.2082E{-3}$ & 0.9973  &$8.2063E{-3}$ & 0.9973   \\
$320$     &$4.2566E{-3}$ & 0.9819  &$4.1071E{-3}$ & 0.9988  &$4.1071E{-3}$ & 0.9989  &$4.1068E{-3}$ & 0.9987  \\
$640$     &$2.2267E{-3}$ & 0.9348  &$2.0544E{-3}$ & 0.9994  &$2.0545E{-3}$ & 0.9994  &$2.0544E{-3}$ & 0.9994  \\
$1280$    &$1.1795E{-3}$ & 0.9167  &$1.0274E{-3}$ & 0.9997  &$1.0274E{-3}$ & 0.9997  &$1.0274E{-3}$ & 0.9996   \\
$2560$    &$6.6893E{-4}$ & 0.8182  &$5.1377E{-4}$ & 0.9998  &$5.1377E{-4}$ & 0.9998  &$5.1377E{-4}$ & 0.9998   \\
\hline
\end{tabular}
\end{footnotesize}
\end{center}
\label{table: IFE solution H1 error 1 10}
\end{table}

\begin{table}[ht]
\caption{Comparison of $\|u_h - u\|_{L^2(\Omega)}$ for different IFE methods with $\beta^- = 1$, $\beta^+ = 10$.}
\vspace{-5mm}
\begin{center}
\begin{footnotesize}
\begin{tabular}{|c|cc|cc|cc|cc|}
\hline
& \multicolumn{2}{c|}{Classic IFE}
& \multicolumn{2}{c|} {SPP IFE}
& \multicolumn{2}{c|} {IPP IFE}
& \multicolumn{2}{c|} {NPP IFE}\\
\hline
$N$
& $\|\cdot\|_{L^2}$ & rate
& $\|\cdot\|_{L^2}$ & rate
& $\|\cdot\|_{L^2}$ & rate
& $\|\cdot\|_{L^2}$ & rate  \\
\hline
$20$      &$4.3003E{-3}$ &         &$4.2945E{-3}$ &         &$4.2989E{-3}$ &         &$4.2869E{-3}$ &        \\
$40$      &$1.0622E{-3}$ & 2.0174  &$1.0749E{-3}$ & 1.9983  &$1.0745E{-3}$ &  2.0003 &$1.0626E{-3}$ & 2.0124 \\
$80$      &$2.6196E{-4}$ & 2.0196  &$2.6833E{-4}$ & 2.0021  &$2.6797E{-4}$ &  2.0036 &$2.6440E{-4}$ & 2.0067 \\
$160$     &$6.4952E{-5}$ & 2.0119  &$6.7047E{-5}$ & 2.0008  &$6.6872E{-5}$ &  2.0026 &$6.5876E{-5}$ & 2.0049 \\
$320$     &$1.6311E{-5}$ & 1.9935  &$1.6829E{-5}$ & 1.9942  &$1.6794E{-5}$ &  1.9935 &$1.6594E{-5}$ & 1.9891 \\
$640$     &$4.4482E{-6}$ & 1.8746  &$4.2038E{-6}$ & 2.0012  &$4.1934E{-6}$ &  2.0017 &$4.1383E{-6}$ & 2.0035 \\
$1280$    &$1.4445E{-6}$ & 1.6226  &$1.0501E{-6}$ & 2.0012  &$1.0472E{-6}$ &  2.0015 &$1.0336E{-6}$ & 2.0014 \\
$2560$    &$6.7593E{-7}$ & 1.0956  &$2.6254E{-7}$ & 1.9999  &$2.6149E{-7}$ &  2.0017 &$2.5821E{-7}$ & 2.0357 \\
\hline
\end{tabular}
\end{footnotesize}
\end{center}
\label{table: IFE solution L2 error 1 10}
\end{table}

\begin{table}[ht]
\caption{Comparison of $\|u_h - u\|_{L^\infty(\Omega)}$ for different IFE methods with $\beta^- = 1$, $\beta^+ = 10$.}
\vspace{-5mm}
\begin{center}
\begin{footnotesize}
\begin{tabular}{|c|cc|cc|cc|cc|}
\hline
& \multicolumn{2}{c|}{Classic IFE}
& \multicolumn{2}{c|} {SPP IFE}
& \multicolumn{2}{c|} {IPP IFE}
& \multicolumn{2}{c|} {NPP IFE}\\
\hline
$N$
& $\|\cdot\|_{L^\infty}$ & rate
& $\|\cdot\|_{L^\infty}$ & rate
& $\|\cdot\|_{L^\infty}$ & rate
& $\|\cdot\|_{L^\infty}$ & rate  \\
 \hline
$20$      &$1.0969E{-3}$ &         &$1.3680E{-3}$ &         &$1.3785E{-3}$ &         &$1.0082E{-3}$ &        \\
$40$      &$5.4748E{-4}$ & 1.0026  &$3.9775E{-4}$ & 1.7822  &$3.9769E{-4}$ & 1.7934  &$1.9172E{-4}$ & 2.3947 \\
$80$      &$5.0812E{-4}$ & 0.1077  &$1.0601E{-4}$ & 1.9077  &$1.0582E{-4}$ & 1.9100  &$5.4491E{-5}$ & 1.8149 \\
$160$     &$2.2635E{-4}$ & 1.1667  &$3.1598E{-5}$ & 1.7463  &$3.1217E{-5}$ & 1.7612  &$1.4045E{-5}$ & 1.9559 \\
$320$     &$1.2290E{-4}$ & 0.8811  &$7.0324E{-6}$ & 2.1677  &$6.9364E{-6}$ & 2.1701  &$3.5092E{-6}$ & 2.0009 \\
$640$     &$7.0810E{-5}$ & 0.7954  &$1.9288E{-6}$ & 1.8664  &$1.9213E{-6}$ & 1.8521  &$9.1942E{-7}$ & 1.9324 \\
$1280$    &$3.4111E{-5}$ & 1.0537  &$5.0505E{-7}$ & 1.9332  &$4.9869E{-7}$ & 1.9459  &$2.2932E{-7}$ & 2.0034 \\
$2560$    &$1.7815E{-5}$ & 0.9371  &$1.0199E{-7}$ & 2.3080  &$1.0075E{-7}$ & 2.3074  &$5.9784E{-8}$ & 1.9395 \\
\hline
\end{tabular}
\end{footnotesize}
\end{center}
\label{table: IFE solution infinity error 1 10}
\end{table}

It is known that the point-wise accuracy of the classic IFE methods in the literature is usually quite poor around the interface and we suspect the discontinuity of the IFE functions across interface edges is the main cause for this shortcoming. With the penalty to control the discontinuity in IFE functions, a partially penalized IFE method has the potential to produce better point-wise approximations. To demonstrate this, we plot errors of a classic bilinear IFE solution and a NPP IFE solution in Figure \ref{fig: infinity norm error comparison}. The IFE solutions in these plots are generated on the mesh with $80\times 80$ rectangles. From the plot on the left, we can easily see that the classic IFE solution has much larger errors in the vicinity of the interface. The plot on the right shows that the magnitude of the error in the NPP IFE solution is much smaller uniformly over the whole solution domain. This advantage is also observed for other partially penalized IFE solutions.

\begin{figure}[ht]
\begin{center}
  \includegraphics[width=0.45\textwidth]{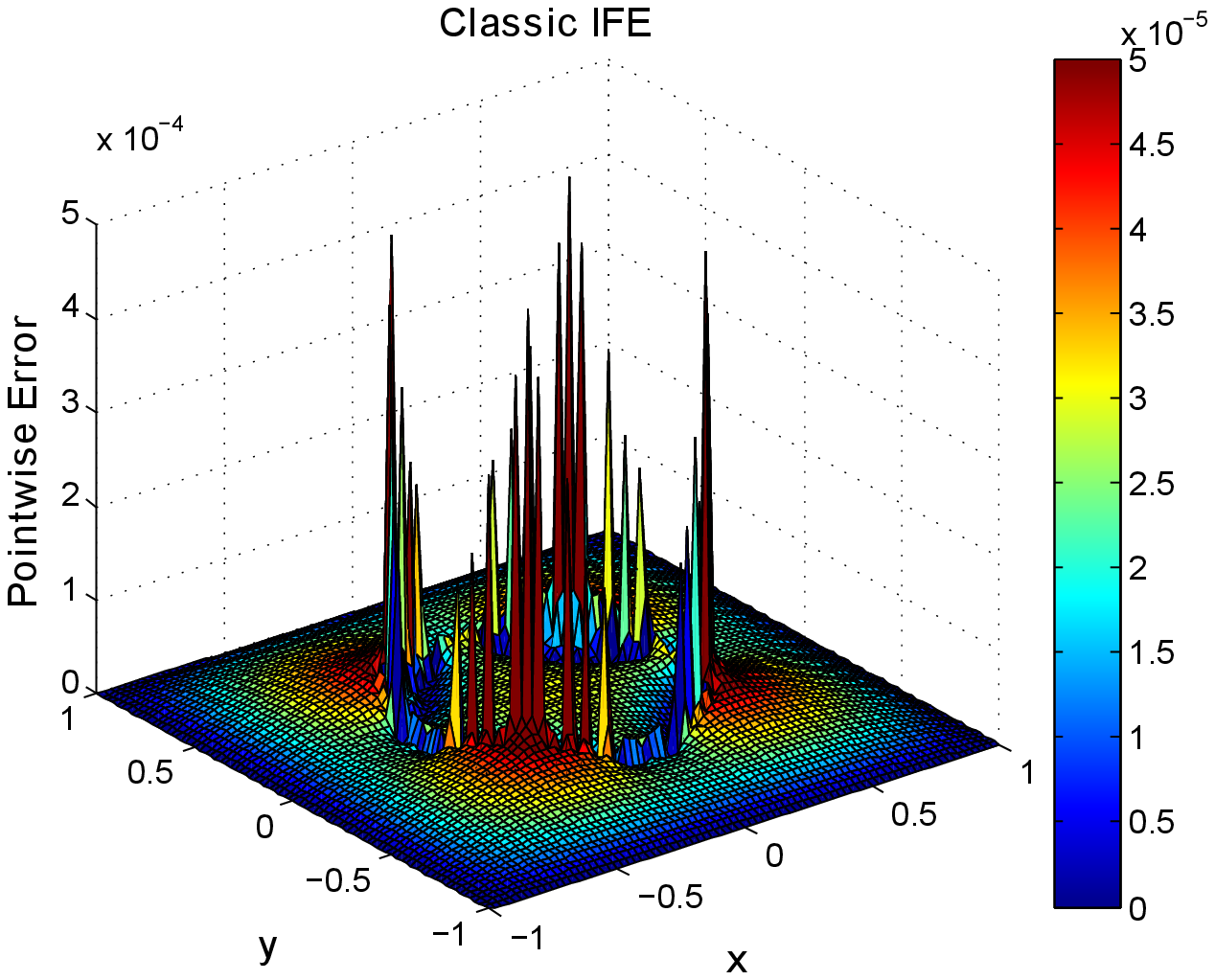}\quad
  \includegraphics[width=0.45\textwidth]{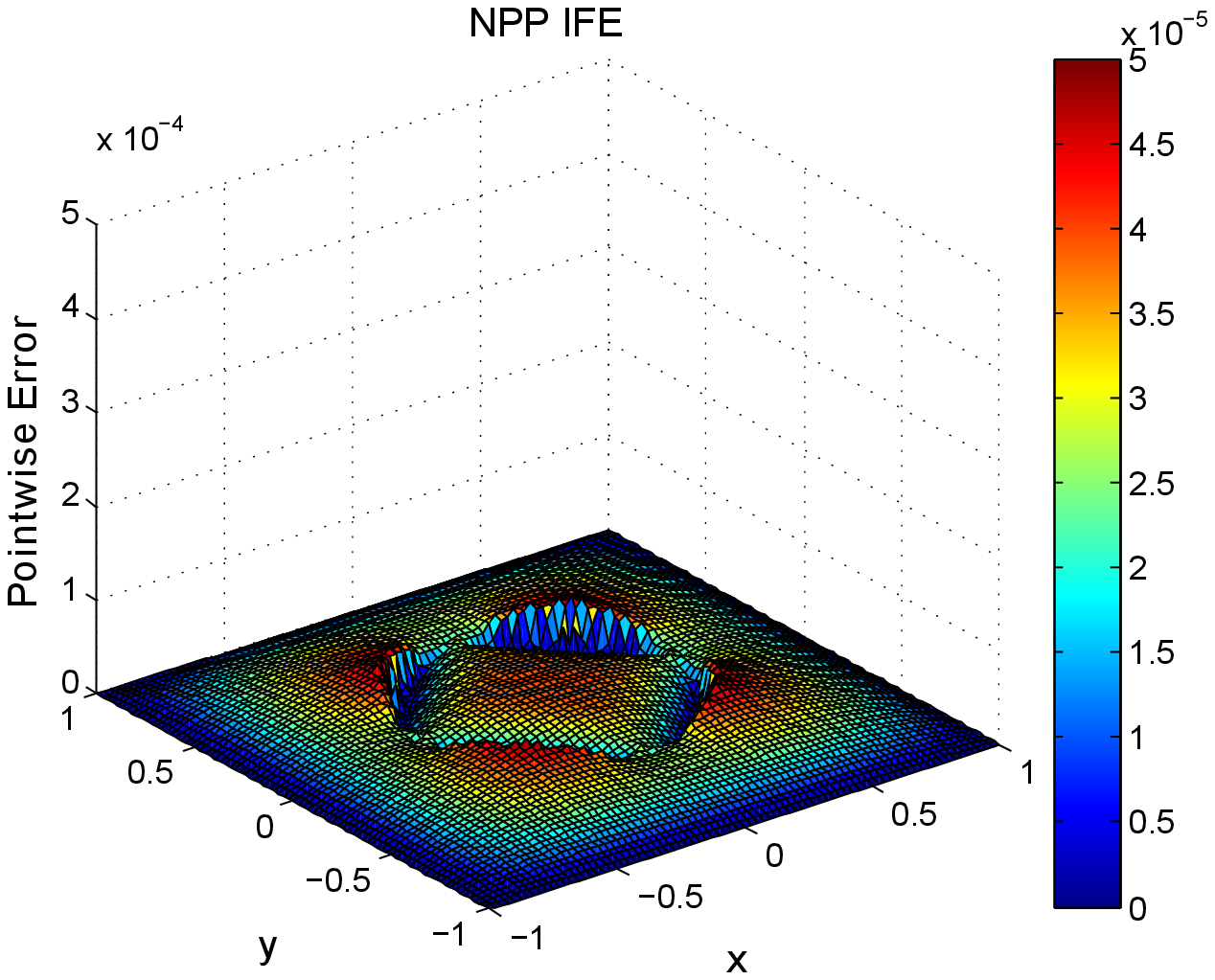}
  \end{center}
\caption{A comparison of point-wise errors of a classic bilinear IFE solution and a nonsymmetric partially penalized IFE solution.
  The coefficient in the interface problem has a moderate jump: $(\beta^-,\beta^+) = (1,10)$.}
  \label{fig: infinity norm error comparison}
\end{figure}

\begin{appendices}

\section{A proof of Remark \ref{lem:interp_error_bnd_edge_inf}}\label{remark_1}

We give a proof for the linear IFEs, and the same arguments can be used to show this error bound for
bilinear IFEs.

Without loss of generality, let $K = \bigtriangleup A_1A_2A_3$ be an interface triangle whose vertices and interface intersection points are given in \eqref{eq: triangle vertices} and \eqref{eq: triangle DE}
with $A_1 \in K^+$. Also we only discuss $B = \overline{A_1A_2}$, the estimate on the other
interface edge can be established similarly.

By Lemma 3.3 and Lemma 3.4 in \cite{ZLi_TLin_YLin_RRogers_linear_IFE}, for every $X \in \overline{DA_2}$, we have
\begin{align}
(I_{h}u(X) - u(X))_p &= \big(N^-(D) - N_{\overline{DE}}\big)\nabla u^-(X)(A_1-D) \pderiv{\phi_1(X)}{p}
\nonumber \\
&  +I_1(X) \pderiv{\phi_1(X)}{p} + I_2(X)\pderiv{\phi_2(X)}{p} + I_3(X) \pderiv{\phi_3(X)}{p},~~p = x, y,
\label{eq:interp_error_bnd_edge_inf_1}
\end{align}
where $N^-(D)$, $N_{\overline{DE}}$, $\rho$,  ${\bf n}(X)$, and ${\bf n}(\overline{DE})$ are defined the same as in the proof of Lemma \ref{lem:interp_error_bnd_edge}.
The quantities $I_i(X)$, $i=1,2,3,$ are given in \eqref{eq: I1} and \eqref{eq: I2I3}.
By Lemma 3.1 and Theorem 2.4 of \cite{ZLi_TLin_YLin_RRogers_linear_IFE}, we have
\begin{align}
& \int_{\overline{DA_2}}\left(\big(N^-(D) - N_{\overline{DE}}\big)\nabla u^-(X)(A_1-D) \pderiv{\phi_1(X)}{p}\right)^2dX \leq C h^3 \norm{u}_{W^{2,\infty}(\Omega^-)}^2
\label{eq:interp_error_bnd_edge_inf_2}
\end{align}
for $p=x,y$. By direct calculations we have for $i = 1, 2, 3, p = x, y$
\begin{align*}
&\abs{\oderiv{\nabla u^-}{t}(tD + (1-t)X)\cdot(A_1-X)} \leq \big(\abs{u_{xx}^-(tD + (1-t)X)}
+ \abs{u_{xy}^-(tD + (1-t)X)} \\
& \hspace{1.7in} + \abs{u_{yx}^-(tD + (1-t)X)} + \abs{u_{yy}^-(tD + (1-t)X)}\big)h^2, \\
&\abs{\oderivm{u^s}{t}{2}(tA_i + (1-t)X)} \leq \big(\abs{u_{xx}^s(tA_i + (1-t)X)}
+ \abs{u_{xy}^s(tA_i + (1-t)X)} \\
& \hspace{1.7in} + \abs{u_{yx}^s(tA_i + (1-t)X)} + \abs{u_{yy}^s(tA_i + (1-t)X)}\big)h^2.
\end{align*}
By these estimates and Theorem 2.4 of \cite{ZLi_TLin_YLin_RRogers_linear_IFE}, we have
\begin{align}
\int_{\overline{DA_2}}\left(I_i(X)\pderiv{\phi_i(X)}{p}\right)^2 dX \leq Ch^3
\norm{u}_{\tW^{2,\infty}(\Omega)}^2, ~~ \label{eq:interp_error_bnd_edge_inf_3}
\end{align}
Then, using \eqref{eq:interp_error_bnd_edge_inf_2} and applying \eqref{eq:interp_error_bnd_edge_inf_2} and \eqref{eq:interp_error_bnd_edge_inf_3} we have
\begin{align}
\int_{\overline{DA_2}}\big((I_{h}u(X) - u(X))_p\big)^2 dX \leq  Ch^3\norm{u}_{\tW^{2,\infty}(\Omega)}^2,
~~p = x, y. \label{eq:interp_error_bnd_edge_inf_4}
\end{align}
Similar arguments can be used to show
\begin{align}
\int_{\overline{A_1D}}\big((I_{h}u(X) - u(X))_p\big)^2 dX \leq  Ch^3\norm{u}_{\tW^{2,\infty}(\Omega)}^2,
~~p = x, y. \label{eq:interp_error_bnd_edge_inf_5}
\end{align}
Finally, the estimate \eqref{eq:interp_error_bnd_edge_inf} on the interface edge
$B = \overline{A_1A_2}$ follows from
\eqref{eq:interp_error_bnd_edge_inf_4} and \eqref{eq:interp_error_bnd_edge_inf_5}.

\end{appendices}
\bibliographystyle{abbrv}

\end{document}